\documentclass[12pt,letterpaper,english]{amsart}
\usepackage{lmodern}
\usepackage{helvet}

\usepackage[T1]{fontenc}
\usepackage[latin9]{inputenc}
\usepackage{mathrsfs}
\usepackage{amsthm}
\usepackage{amstext}
\usepackage{amssymb}
\usepackage{esint}
\usepackage{graphicx}

\makeatletter

\pdfpageheight\paperheight
\pdfpagewidth\paperwidth

\newcommand{\noun}[1]{\textsc{#1}}

\numberwithin{equation}{section}
\numberwithin{figure}{section}
\usepackage{enumitem}		
\theoremstyle{plain}
\newtheorem{thm}{\protect\theoremname}[section]
\theoremstyle{remark}
\newtheorem{rem}[thm]{\protect\remarkname}
\theoremstyle{definition}
\newtheorem{defn}[thm]{\protect\definitionname}
\theoremstyle{definition}
\newtheorem{example}[thm]{\protect\examplename}
\theoremstyle{plain}
\newtheorem{cor}[thm]{\protect\corollaryname}
\theoremstyle{plain}
\newtheorem{prop}[thm]{Proposition}

\usepackage{amsfonts}
\usepackage{amssymb}
\usepackage[all]{xy}
\usepackage{array}
\usepackage{lmodern}
\usepackage[T1]{fontenc}

\def\d{\partial}
\def\uo{\underline{0}}
\def\ul{\underline{\ell}}
\def\un{\underline{n}}
\def\uk{\underline{\kappa}}
\def\utt{\underline{\tau}}
\def\ut{\underline{t}}
\def\uv{\underline{v}}
\def\um{\underline{m}}
\def\uy{\underline{y}}
\def\ux{\underline{x}}
\def\QQ{\mathbb{Q}}
\def\RR{\mathbb{R}}
\def\CC{\mathbb{C}}

\def\ZZ{\mathbb{Z}}
\def\PP{\mathbb{P}}

\def\DD{\mathbb{D}}
\def\GG{\mathbb{G}}
\def\uz{\underline{z}}
\def\ut{\underline{t}}
\def\W{\mathcal{W}}

\def\VV{\mathbb{V}}

\def\F{\mathcal{F}}

\def\V{\mathcal{V}}

\def\X{\mathcal{X}}

\def\M{\mathcal{M}}
\def\H{\mathcal{H}}

\def\NN{\mathbb{N}}
\def\D{\mathcal{D}}

\def\vf{\varphi}

\def\l{\lambda}

\def\t{\lambda}
\def\k12{\mathcal{K}_{\lambda_1,\lambda_2}}
\def\tk12{\tilde{\mathcal{K}}_{\lambda_1,\lambda_2}}
\def\ck12{\check{\mathcal{K}}_{\lambda_1,\lambda_2}}

\def\ay{\mathbf{i}}
\def\sx{\mathsf{X}}
\def\rx{\mathrm{X}}

\theoremstyle{definition}

\theoremstyle{definition}

\theoremstyle{definition}
\newtheorem*{thx}{Acknowledgments}

\makeatother

\usepackage{babel}
  \providecommand{\corollaryname}{Corollary}
  \providecommand{\definitionname}{Definition}
  \providecommand{\remarkname}{Remark}
\providecommand{\theoremname}{Theorem}
 \providecommand{\examplename}{Example}

\begin{document}

\title{Motivic irrationality proofs}

\author{Matt Kerr}

\begin{abstract}
We exhibit geometric conditions on a family of toric hypersurfaces
under which the value of a canonical normal function at a point of
maximal unipotent monodromy is irrational.
\end{abstract}
\maketitle

\section{Introduction}\label{S1}

The limit of a generalized normal function at a point where the underlying
variation of Hodge structure degenerates, as recently studied in \cite{7K},
turns out to have an unexpected arithmetic application. R. Ap\'ery's
famous proof (see \cite{vdP}) of irrationality of $\zeta(3):=\sum_{k\geq1}k^{-3}$
relies on the existence of rapidly divergent sequences $a_{m}\in\ZZ$,
$b_{m}\in\QQ$ (the latter having denominators of bounded growth)
with $2a_{m}\zeta(3)+b_{m}$ converging rapidly to zero.  (Here we have
swapped the meanings of $a_m$ and $b_m$ in \cite{vdP} and 
multiplied the latter by $-2$.)
Beukers, Peters and Stienstra \cite{Be,BP,Pe,PS} geometrically repackaged
much of the proof, noting for instance that the generating function
$\sum_{m\geq0}a_{m}\lambda^{m}=:A(\lambda)$ records the period of a
holomorphic 2-form on a family of $K3$ surfaces $\{X^{\lambda}\}_{\lambda\in\PP^{1}}$,
hence must satisfy a Picard-Fuchs differential equation $D_{\text{PF}}A(\lambda)=0$.

Behind the remaining details of the irrationality proof lurks a family
of cycles in (algebraic) $K_{3}$ of the $K3$. The associated higher
normal function $\tilde{V}(\lambda)$ has special value $\tilde{V}(0)=-2\zeta(3)$,
and satisfies the inhomogeneous equation $D_{\text{PF}}\tilde{V}(\lambda)=Y(\lambda)$,
where $Y$ denotes the Yukawa coupling. Setting $\sum_{m\geq1}b_{m}\lambda^{m}:=A(\lambda)\tilde{V}(0)-\tilde{V}(\lambda)$,
one deduces from this recurrence relations on the $\{b_{m}\}$ which
(as presented here) give ``half'' of the required bounded denominator
growth. The other ``half'' comes from the \emph{Fermi family} of
$K3$s studied, but not linked to the Ap\'ery proof, in \cite{PS}.
Finally, the behavior of the cycles at singular members
of the family $\{X^{\lambda}\}$ shows that $\tilde{V}(\lambda)$
has no monodromy about the conifold singular fiber closest to $\lambda=0$,
implying the rapid convergence of $2a_{m}\zeta(3)+b_{m}\to0$.

In this paper, we reveal a general criterion for the irrationality
of special values of certain higher normal functions. (An expository,
much less technical account of the proof for $\zeta(3)$ will appear elsewhere.) Given a Laurent
polynomial $\phi(x_{1},\ldots,x_{n})$ with reflexive Newton polytope
$\Delta$, the equation $\phi(\underline{x})=\lambda$ defines a family
of Calabi-Yau hypersurfaces $\{X^{\lambda}\}$ in the toric variety
$\PP_{\Delta}$. Associated to this family is a pure irreducible variation of Hodge structure $\V_{\phi}$ of weight $(n-1)$ over a Zariski open $\mathcal{U}\subset \PP^1$, together with a canonical section $\{\tilde{\omega}^{\lambda}\}$
of the Hodge line bundle $\F^{n-1}\V_{\phi}$. We call $\phi$ \emph{tempered} if the coordinate
symbol $\{x_{1},\ldots,x_{n}\}$ lifts to a family of motivic cohomology
classes $\Xi^{\lambda}\in H^n_{\mathcal{M}}(X^{\lambda},\QQ(n))$ on the family, producing an extension \begin{equation}\label{ext} 0\to \V_{\phi}(n)\to \mathcal{E}_{\phi} \to \QQ(0)_{\mathcal{U}}\to 0\end{equation}of admissible variations of mixed Hodge structure over $\mathcal{U}$. (This temperedness typically holds, for
example, for Landau-Ginzburg models constructed from Minkowski polynomials \cite{dS}.)
Applying a variant of this hypothesis allows us to construct a canonical
truncated higher normal function $\tilde{V}(\lambda)$ on $\PP^{1}\backslash\phi(\RR_{-}^{\times n})$
by pairing the regulator class of $\Xi^{\lambda}$ (i.e., the extension class of \eqref{ext}) with $\tilde{\omega}^{\lambda}$
(see Theorem \ref{th3}).\footnote{This $\tilde{V}$ may be viewed as a period of the extension \emph{dual} to \eqref{ext}.}

To arrange $\tilde{V}(0)\notin\QQ$, we must impose several additional
conditions on $\phi$, roughly as follows (see Theorem \ref{th2} and Prop. \ref{prop}):
\begin{itemize}[leftmargin={*}]
\item the local system of periods of $\tilde{\omega}^{\lambda}$ must be
of rank $n$, admit an isomorphism to its pullback by $\lambda\mapsto C/\lambda$,
and have two mild singularities apart from $0$ and $\infty$, one
of which is very far from $0$;
\item $\phi(-\underline{x})$ has positive integer coefficients, and the
Picard-Fuchs operator associated to $\tilde{\omega}^{\lambda}$ (suitably
normalized) is integral; and
\item a finite ($r:1$) pullback of the family $X^{\lambda}$ can be presented as
a family of toric hypersurfaces in $\PP_{\Diamond}$, where $\Diamond$
is a ``facile'' polytope (Definition \ref{def1}).
\end{itemize}
The role of the last condition is to produce a basis of periods whose
power-series coefficients have the right denominator bounds (see Corollary
\ref{cor1}). This basis is closely tied to mirror symmetry
\cite{HLY} and the Frobenius method \cite{Yo}; for $n\geq4$, Theorem
\ref{th1} uncovers a surprising arithmetic implication of the
Hyperplane Conjecture \cite{HLY,LZ}, which may be of independent interest.  We also remark that, assuming
only temperedness, the higher normal function $\tilde{V}(\lambda)$
can always be written as one of the chain-integral solutions of \cite{HLYZ}.
Moreover, its asymptotics at $\lambda=\infty$ are related to the Frobenius constants and
gamma-classes of \cite{BV,GGI,GI,GZ},
while $\tilde{V}(0)$ may be interpreted as an Ap\'ery constant as
studied in \cite{Ga,Go}.

In the last section, we exhibit Laurent polynomials which satisfy
all these conditions for $n=1$, $2$, and $3$, recovering irrationality
of $\log(1+b^{-1})$ ($b\in\mathbb{N}$), $\zeta(2)$, and $\zeta(3)$.
We also propose relaxations of some of the conditions, together with
specific families of polynomials, for making contact with results
on linear forms in more than one odd zeta value -- for instance \cite{Va},
\cite{Z1}, \cite{BR}, and especially \cite{Br}, whose basic cellular integrals
on $\mathcal{M}_{0,n+3}$ are the power series coefficients of a $\tilde{V}(\lambda)$
as above.

While the results on $\zeta(2)$ and $\zeta(3)$ complete, in a way,
the story begun in \cite{BP}, the reader familiar with those works
will notice (perhaps deplore?) the complete lack of reference to modular
forms in what follows. The omission is strategic, as a weight-$(n-1)$
VHS $\mathcal{V}_{\phi}$ with maximal unipotent and conifold monodromies
\emph{cannot} have a modular parametrization for $n\geq4$. This is,
of course, precisely where we hope to stimulate the search for examples
with Theorem \ref{th2}, starting with the increasingly sophisticated
databases of polytopes, local systems, Calabi-Yau differential operators 
and their geometric realizations \cite{CYclass,ASZ,Fano1,Fano,DM}.
(It is also the setting where, in articles to come, we will extend
the methods developed below in connection with Ap\'ery and gamma
constants of Fano varieties.)

The same reader may be puzzled by our reference to ``splitting up the bound''
on denominators of $b_m$, as in the Ap\'ery $\zeta(3)$ story one simply shows that
$2(L_m)^3 b_m \in \ZZ$, where $L_m := \text{lcm}\{1,2,\ldots,m\}$.
What we are able to show under \emph{general} hypotheses, using
techniques from Hodge theory and mirror symmetry,
is instead that (for some unknown fixed $\varepsilon\in \mathbb{N}$) both
$\varepsilon (L_{rm})^n b_m$ and $\varepsilon (m!)^n b_m$ are
integers for all $m$. Though weaker than the above bound in the Ap\'ery $\zeta(3)$ ($r=2$, $n=3$) case,
this is still enough for an irrationality proof.
With that said, the results of this article are intended to be a narrow proof of concept
for an approach to irrationality proofs via motivic cohomology and Hodge theory, rather than to be optimal as respects
either methodology or hypotheses.  In particular, the removal of the restriction to normal 
functions arising from $K_n$ of Calabi-Yau $(n-1)$-folds will be considered in a subsequent work.

We freely (though infrequently) use the notation and terminology of
regulator currents (e.g. $T_{\Xi},\Omega_{\Xi},R_{\Xi}$) and toric varieties throughout, as reviewed for instance in
$\S\S$1-2 of \cite{DK}.
The heart of this paper is $\S$\ref{S3} (together with the examples in $\S$\ref{S5}), and no more than a quick perusal of the technical $\S$\ref{S2} 
is needed on a first reading.  The difference in notation ($\Diamond$ vs. $\Delta$) for polytopes in $\S$\ref{S2} and $\S$\ref{S3}
is intended to draw a bright line between the family $X^{\lambda}\subset \PP_{\Delta}$ and its facile $r$-cover $\mathrm{X}_{\xi}\subset \PP_{\Diamond}$.
Similarly, before bothering with $\S$\ref{S4} (where the regulator-current calculus has its main appearance)
the reader may prefer to look at how it is used in $\S\S$\ref{SS5.2}-\ref{SS5.3}.

Throughout we shall write $\partial_t$ resp. $\delta_t$ for $\frac{\partial}{\partial t}$ resp. $t\tfrac{\partial}{\partial t}$.

\begin{thx}
It is my pleasure to thank S. Bloch, F. Brown, C. Doran, V. Golyshev,
B. Lian, V. Maillot, F. Rodriguez-Villegas, D. Zagier and W. Zudilin
for helpful conversations and correspondence -- some recent, and some
further back. This work was partially supported by NSF FRG Grant DMS-1361147 and Simons Collaboration Grant 634268.
\end{thx}

\section{Facile polytopes}\label{S2}

Let $\Diamond,\Diamond^{\circ}$ be a dual pair of reflexive polytopes
in $\RR^{n}$, admitting regular projective triangulations $\mathcal{T},\mathcal{T}^{\circ}$.
Take $\Sigma,\Sigma^{\circ}$ to be the fans on these triangulations,
and $\PP_{\Diamond^{\circ}},\PP_{\Diamond}$ (respectively -- note the reversal) the toric
Fano $n$-folds determined by the fans. Write
\[
\mathcal{A}=\mathcal{A}_{\Diamond}:=\Diamond\cap\ZZ^{n}=\{\underline{v}^{(0)}=\underline{0},\,\underline{v}^{(1)},\ldots,\underline{v}^{(N)}\}
\]
for the integer points, and 
\[
\mathbb{L}:=\left\{ \left.\underline{\ell}=\left(\ell_{0}=-\Sigma_{i=1}^{N}\ell_{i},\,\ell_{1},\ldots,\ell_{N}\right)\in\ZZ^{N+1}\right|\Sigma_{i=1}^{N}\ell_{i}\underline{v}^{(i)}=\underline{0}\right\} 
\]
for the lattice of integral relations on them. Let $x_1,\ldots,x_n$ denote the coordinates on $\mathbb{G}_m^n$ and also their extensions to rational functions on $\PP_{\Diamond}$ and $\PP_{\Diamond^{\circ}}$. The irreducible components
$\{D_{i}\}_{i=1}^{N}$ of $\PP_{\Diamond^{\circ}}\backslash\mathbb{G}_{m}^{n}=\mathbb{D}_{\Diamond^{\circ}}$
(with $\deg_{D_{i}}(x_{k})=v_{k}^{(i)}$) generate $H^{2}(\PP_{\Diamond^{\circ}},\ZZ)$,
and $K_{\PP_{\Diamond^{\circ}}}=-\sum_{i=1}^{N}D_{i}=:D_{0}$.

Assume that the Mori cone $\mathcal{M}=\mathcal{M}_{\Diamond^{\circ}}\subset H_{2}(\PP_{\Diamond^{\circ}},\RR)$
(of classes of effective curves) is regular simplicial, so that $\mathbb{M}:=\mathcal{M}\cap H_{2}(\PP_{\Diamond^{\circ}},\ZZ)=\ZZ_{\geq0}\langle C_{1},\ldots,C_{M=N-n}\rangle.$
Write $C_{j}\mapsto\underline{\ell}^{(j)}=C_{j}\cdot\underline{D}$
for the images under 
\[
H_{2}(\PP_{\Diamond^{\circ}},\ZZ)\overset{\cong}{\to}\mathbb{L}\overset{\cong}{\to}\text{Hom}\left(H^{2}(\PP_{\Diamond^{\circ}},\ZZ),\ZZ\right),
\]
so that the dual (nef) cone $\overline{\mathcal{K}}\subset H^{2}(\PP_{\Diamond^{\circ}},\RR)$
has $\mathbb{K}:=\overline{\mathcal{K}}\cap H^{2}(\PP_{\Diamond^{\circ}},\ZZ)=\ZZ_{\geq0}\langle J_{1},\ldots,J_{M}\rangle$
with $J_{k}\cdot C_{j}=\delta_{kj}$ and $D_{i}=\sum_{k=1}^{M}\ell_{i}^{(k)}J_{k}$.
(To compute the $\underline{\ell}^{(k)}$, one may use primitive collections
as in \cite{LZ}.)

Let $F_{\underline{\mathtt{a}}}(\underline{x}):=\sum_{i=0}^{N}\mathtt{a}_{i}\underline{x}^{\underline{v}^{(i)}}$,
and $X_{\underline{\mathtt{a}}}^{\times}:=\{F_{\underline{\mathtt{a}}}(\underline{x})=\underline{0}\}\subset\mathbb{G}_{m}^{n}$
with (CY $(n-1)$-fold) Zariski closure $X_{\underline{\mathtt{a}}}\overset{\imath}{\hookrightarrow}\PP_{\Diamond}$;
take $X^{\circ}\subset\PP_{\Diamond^{\circ}}$ to be any smooth, $\Diamond^{\circ}$-regular anticanonical
hypersurface. We consider the $\mathbb{A}$\emph{-periods} ($\mathbb{A}=\ZZ,\QQ,\text{ or }\CC$) of
\[
\omega_{\underline{\mathtt{a}}}=\text{Res}_{X_{\underline{\mathtt{a}}}}\Omega_{\underline{\mathtt{a}}}=\tfrac{\mathtt{a}_{0}}{(2\pi{\bf i})^{n-1}}\text{Res}_{X_{\underline{\mathtt{a}}}}\left(\tfrac{dx_{1}/x_{1}\wedge\cdots\wedge dx_{n}/x_{n}}{F_{\underline{\mathtt{a}}}(\underline{x})}\right)\in\Omega^{n-1}(X_{\underline{\mathtt{a}}}),
\]
which are the multivalued functions $\pi_{\gamma}(\underline{\mathtt{a}})=\int_{\gamma}\omega_{\underline{\mathtt{a}}}$ ($\gamma\in H_{n-1}(X_{\underline{\mathtt{a}}},\mathbb{A})$)
defined on a Zariski open in the affine $(N+1)$-space with coordinates $(\mathtt{a}_0,\ldots,\mathtt{a}_N)$.  These periods are known to solve the GKZ system
\[
\tau_{\text{GKZ}}^{\Diamond}:\;\;\begin{array}{c}
\left\{ \left(\prod_{\ell_{i}>0}\partial_{\mathtt{a}_{i}}^{\ell_{i}}-\prod_{\ell_{i}<0}\partial_{\mathtt{a}_{i}}^{-\ell_{i}}\right)\tfrac{1}{\mathtt{a}_{0}}\right\} _{\ell\in\mathbb{L}},\\
\left\{ \left(\Sigma_{i=0}^{N}v_{j}^{(i)}\delta_{\mathtt{a}_{i}}\right)\tfrac{1}{\mathtt{a}_{0}}\right\} _{j=1,\ldots,n},\;\left(\Sigma_{i=0}^{N}\delta_{\mathtt{a}_{i}}+1\right)\tfrac{1}{\mathtt{a}_{0}},
\end{array}
\]
whose remaining solutions are the other integrals of $\frac{1}{2\pi{\bf i}}\Omega_{\underline{\mathtt{a}}}$
over \emph{relative} cycles in $H_{n}\left(\PP_{\Diamond}\backslash X_{\underline{\mathtt{a}}},\DD_{\Diamond}\backslash\DD_{\Diamond}\cap X_{\underline{\mathtt{a}}};\CC\right)$
\cite{HLYZ}.  We are interested in the behavior of the $\pi_{\gamma}$ in the \emph{large complex-structure limit} (LCSL), i.e., where
the $t_{k}:=\underline{\mathtt{a}}^{\underline{\ell}^{(k)}}$ are sufficiently small. Our conclusion will be Theorem \ref{th1} below.

A formula for the solutions to $\tau_{\text{GKZ}}^{\Diamond}$ in the
LCSL was given by \cite{HLY}: writing $\tau_{i}:=\frac{\log(t_{i})}{2\pi{\bf i}}$,
$\mathcal{O}:=\CC[[\underline{t}]][\underline{\tau}]$, they are precisely
the functions $\psi(\mathcal{B}_{\Diamond})\in\mathcal{O}$ where $\psi\in H^{\bullet}(\PP_{\Diamond^{\circ}},\CC)^{\vee}$
and
\[
\mathcal{B}_{\Diamond}:=\sum_{\underline{\ell}\in\mathbb{M}}\mathcal{B}_{\underline{\ell}}(\underline{D})\underline{\mathtt{a}}^{\underline{\ell}+\underline{D}}=\sum_{\underline{n}\in\ZZ_{\geq0}^{M}}\mathbb{B}_{\underline{n}}(\underline{J})\underline{t}^{\underline{n}+\underline{J}}\in H^{\bullet}(\PP_{\Diamond^{\circ}},\mathcal{O}),
\]
with $\mathbb{B}_{\underline{n}}(\underline{J}):=\mathcal{B}_{\sum n_{k}\underline{\ell}^{(k)}}\left(\sum\underline{\ell}^{(k)}J_{k}\right)$
and \begin{equation} \label{3-1}
\mathcal{B}_{\underline{\ell}}(\underline{D}) :=\frac{\prod_{i=1:\ell_i <0}^N D_i (D_i -1 )\cdots (D_i + \ell_i + 1)}{\prod_{i=1:\ell_i >0}^N (D_i + 1)\cdots (D_i + \ell_i )}\times (D_0 -1)\cdots (D_0 + \ell_0).
\end{equation}According to the \emph{Hyperplane Conjecture} \cite{HLY,LZ}, $\psi(\mathcal{B}_{\Diamond})$
is a $\CC$-period (in the above sense) precisely when $\psi$ belongs
to 
\[
\text{im}\left\{ \imath_{*}^{\bullet}:\, H^{\bullet}(X^{\circ},\CC)^{\vee}\to H^{\bullet}(\PP_{\Diamond^{\circ}},\CC)^{\vee}\right\} .
\]

More explicitly, for each $\underline{\kappa}:=(\kappa_{1},\ldots,\kappa_{M})\in\ZZ_{\geq0}^{M}$
with $|\underline{\kappa}|:=\sum\kappa_{j}\leq n$, we compute
\[
B^{\underline{\kappa}}(\underline{t},\underline{\tau}):=\left.\tfrac{1}{(2\pi{\bf i})^{|\underline{\kappa}|}}\left(\partial_{J_{1}}^{\kappa_{1}}\cdots\partial_{J_{M}}^{\kappa_{M}}\mathcal{B}_{\Diamond}\right)\right|_{\underline{J}=\underline{0}}=\sum_{\underline{\kappa}'+\underline{\kappa}''=\underline{\kappa}}(2\pi{\bf i})^{-|\underline{\kappa}''|}\underline{\tau}^{\underline{\kappa}'}\sum_{\underline{n}\in\ZZ_{\geq0}^{M}}b_{\underline{n}}^{\underline{\kappa}''}\underline{t}^{\underline{n}}
\]
where $b_{\underline{n}}^{\underline{\kappa}''}=\left(\partial_{J_{1}}^{\kappa_{1}''}\cdots\partial_{J_{M}}^{\kappa_{M}''}\mathbb{B}_{\underline{n}}\right)\left(\underline{0}\right).$
Given bases $\{\psi_{\ell}^{r}\}\subset H_{2r}(\PP_{\Diamond^{\circ}},\ZZ)$
resp. $\{\hat{\psi}_{\ell}^{r}\}\subset\text{im}(\imath_{*}^{2r})_{\ZZ}$ for each $r$,
we obtain $\CC$-bases for the solutions to $\tau_{\text{GKZ}}^{\Diamond}$
resp. for the $\CC$-periods (assuming the Hyperplane Conjecture)
which are $\ZZ$-linear combinations of the $\{B^{\underline{\kappa}}\}$.
That is, writing $$B^{\underline{0}}(\underline{t})=\sum b_{\underline{n}}^{\underline{0}}\underline{t}^{\underline{n}}=:\sum a_{\underline{n}}\underline{t}^{\underline{n}}=:A(\underline{t})$$ for the period holomorphic in $\underline{t}$,
we have\begin{equation} \label{3-2}
\overset{(\wedge)}{\psi}{}_{\ell}^r (\mathcal{B}_{\Diamond}) = A(\underline{t})\overset{(\wedge)}{P}{}_{\ell}^r (\underline{\tau}) + \sum_{|\underline{\kappa}'|<r} (2\pi{\bf i})^{|\underline{\kappa}'|-r} \underline{\tau}^{\underline{\kappa}'} \sum_{\underline{n}} \left( \sum_{|\underline{\kappa}''|=r-|\underline{\kappa}'|} \overset{(\wedge)}{c}{}_{\underline{\kappa}',\underline{\kappa}''} b_{\underline{n}}^{\underline{\kappa}''} \right) \underline{t}^{\underline{n}}
\end{equation}where $\overset{(\wedge)}{P}{}_{\ell}^{r}$ are $\ZZ$-homogeneous
polynomials of degree $r$ and $\overset{(\wedge)}{c}{}_{\underline{\kappa}',\underline{\kappa}''}\in\ZZ.$
\begin{rem}
(i) The full assertion for the $\CC$-periods holds without the Hyperplane
Conjecture for $r\leq\min\left\{ \tfrac{n-1}{2},1\right\} .$ Writing
$\mathcal{A}'\subset\mathcal{A}$ for the points (if any) interior
to facets, the periods cannot depend on the $\{\log(\mathtt{a}_{i})\}_{\underline{v}^{(i)}\in\mathcal{A}'}$
because taking $\mathtt{a}_{i}\to0$ does not make $X_{\underline{\mathtt{a}}}$ singular.
Moreover, $X^{\circ}$ avoids the corresponding (exceptional) $\{D_{i}\}$;
so there are $M-|\mathcal{A}'|$ independent $\{\hat{\psi}_{\ell}^{1}\}$
(with leading terms $A(\underline{t})\,\times$ the $M-|\mathcal{A}'|$
independent linear combinations of the $\{\tau_{j}\}$ with no such
$\log(\mathtt{a}_{i})$'s). But these must all be periods, since by mirror
symmetry (see \cite{AGM} and \cite[Thm. 6.9]{Ir}) there are $h_{\text{van}}^{n-2,1}(X_{\underline{\mathtt{a}}})=h_{\text{tor}}^{1,1}(X^{\circ})=M-|\mathcal{A}'|$
independent periods of degree $1$ in $\underline{\tau}$.

(ii) Denote the (unipotent) local monodromies in $t_i$ by $T_i$ and their logarithms by $N_i$.  Applying $r$ of these $\{N_i\}$ to a $\ZZ$-period
with ``$\log^{r}(\underline{t})$'' leading term must yield a $\ZZ$-multiple
of $A(\underline{t})$. So (a fixed integer multiple of) this period
must be a $\ZZ$-linear combination of the $\{\psi_{\ell}^{r}(\mathcal{B}_{\Diamond})\}$
plus a $\CC$-linear combination of the $\{\psi_{\ell}^{r'}(\mathcal{B}_{\Diamond})\}_{r'<r}$.
If all of the $\CC$-linear combinations that can appear are themselves
$\CC$-periods, they can be subtracted off. So by (i), one in fact has a basis
of $\CC$-periods of the form \eqref{3-2} for $r\leq\min\{\tfrac{n+1}{2},2\}.$
\end{rem}
Call $\underline{\ell}(\underline{n}):=\sum n_{k}\underline{\ell}^{(k)}$
\emph{effective} (resp. \emph{quasi-effective}) if all $\ell_{i}(\underline{n})$
($i>0$) are $\geq0$ (resp. at most one $<0$). Clearly the 
\[
a_{\underline{n}}=\left\{ \begin{array}{cc}
\frac{(-\ell_{0}(\underline{n}))!}{\prod_{i>0}\ell_{i}(\underline{n})!} & \underline{\ell}(\underline{n})\text{ effective}\\
0 & \text{otherwise,}
\end{array}\right.
\]
being multinomial coefficients, are all integers. Now using \eqref{3-1}:
\begin{itemize}[leftmargin={*}]
\item If $\ul(\un)$ is effective, then the $ $$\left.\left(\d_{D_{0}}^{r_{0}}\cdots\d_{D_{N}}^{r_{N}}\mathcal{B}_{\ul(\un)}\right)(\underline{0})\right/a_{\un}$
are $\ZZ$-linear combinations of products of the form $\prod_{i=0}^N \prod_{j=1}^{|\ell_i (\underline{n})|} j^{-p^{(i)}_j}$, where $\underline{p}^{(i)} \in \ZZ_{\geq 0}^{|\ell_i(\underline{n})|}$ and $|\underline{p}^{(i)}|=r_i$.
Since $b_{\un}^{\uk''}/a_{\un}$ is a $\ZZ$-linear combination of
these with $|\underline{r}|=|\uk''|$, $b_{\un}^{\uk''}$ can be written
as $\tfrac{P}{Q}$ ($P,Q\in\ZZ$) with $Q\mid (L_{|\ell_{0}(\un)|})^{|\uk''|}$,
where $$L_{s}:=\text{lcm}\{1,\ldots,s\} .$$
\item If $\ul(\un)$ is quasi-effective, with (say) $\ell_{1}(\un)<0$,
then $\left(\d_{D_{1}}\mathcal{B}_{\ul(\un)}\right)(\uo)=$
\[
\frac{(-\ell_{0}(\un))!\,(-\ell_{1}(\un)-1)!}{\prod_{i>1}\ell_{i}(\un)!}=\left.\frac{(\Sigma_{i>1}\ell_{i}(\un))!}{\Pi_{i>1}\ell_{i}(\un)!}\right/\frac{(\Sigma_{i>1}\ell_{i}(\un))!}{(-\ell_{0}(\un))!\,(-\ell_{1}(\un)-1)!}
\]
is the quotient of a multinomial coefficient by an integer of the
form $A!/(B-1)!(A-B)!$, which always divides $L_{A}$ \cite[Thm. 3(i)]{Na}. Repeated differentiation
as above now shows that $b_{\un}^{\uk''}=P/Q$ with $Q\mid (L_{|\ul(\un)|^{+}})^{|\uk''|}$,
where $|\ul|^{+} :=\sum_{i:\,\ell_{i}>0}\ell_{i}\;(\geq-\ell_{0}).$\end{itemize}
\begin{defn}\label{def1}
A reflexive polytope $\Diamond\subset\RR^{n}$ is \emph{facile} if:\\
$\bullet$ $\Diamond,\Diamond^{\circ}$ admit regular projective triangulations;\\
$\bullet$ the Mori cone $\mathcal{M}_{\Diamond^{\circ}}$ is regular
simplicial, with generators $\ul^{(k)}$;\\
$\bullet$ $\ul(\un):=\sum n_{k}\ul^{(k)}$ is quasi-effective for
each $\un\in\ZZ_{\geq0}^{M}$; and\\
$\bullet$ $n\leq3$, or the Hyperplane Conjecture holds for $\Diamond$.
\end{defn}
\begin{example}
The convex hulls of $\{(1,0),\, (0,1),\, (1,1),\, (-1,-1)\}$ in $\RR^2$ and $\{(\pm1,0,0),\, (0,\pm1,0),\, (0,0,\pm1)\}$ in $\RR^3$ are facile. These polytopes underlie the ``facile 2-coverings'' (terminology defined at the beginning of $\S$\ref{S3}) in $\S\S$\ref{SS5.2}-\ref{SS5.3}, which compensate for the fact that the two Ap\'ery polytopes are not facile.
\end{example}
Now let $W_{\bullet}$ denote the monodromy weight filtration for
the large complex structure limit (i.e. for the sum of the monodromy logarithms), with 
\[
h_{i}:=\text{rk}\left(\text{Gr}_{2i}^{W}H^{n-1}(X_{\underline{\mathtt{a}}})\right)=\text{rk}\left(\text{Gr}_{-2i}^{W}H_{n-1}(X_{\underline{\mathtt{a}}})\right)=\text{rk}\left(H^{i,n-i-1}(X_{\underline{\mathtt{a}}})\right)
\]
($\underline{\mathtt{a}}$ very general), which is $1$ for $i=0$ or $n-1$.  We have shown:
\begin{thm}\label{th1} 
If $\Diamond$ is facile, there exists a basis of $H_{n-1}(X_{\underline{\mathtt{a}}},\CC)$,
of the form $\gamma_{\mu}^{r}\in W_{2(r-n+1)}H_{n-1}(X_{\underline{\mathtt{a}}},\CC)$
$(r=0,\ldots,n-1;\,\mu=1,\ldots,h_{r})$, with periods \begin{equation} \label{5-1}
\pi^r_{\mu} := \int_{\gamma^r_{\mu}} \omega_{\underline{\mathtt{a}}} = \sum_{|\uk|=r} c_{\mu,\uk}^r \utt^{\uk} \sum a_{\un} \ut^{\un} + \sum_{|\uk|<r} \tfrac{1}{(2\pi\mathbf{i})^{r-|\uk|}}\utt^{\uk}\sum_{\un\in\ZZ^M_{\geq 0}} \beta^{r,\uk}_{\mu,\un} \ut^{\un} ,
\end{equation}where each $a_{\un}\in\ZZ$, each $c_{\mu,\uk}^{r}\in\ZZ$, each $\beta_{\mu,\un}^{r,\uk}$ is of the form $\frac{P}{Q}$
$(P,Q\in\ZZ)$ with $Q\mid (L_{|\ul(\un)|^{+}})^{r-|\uk|}$, and $($for
each $r,\mu)$ some $c_{\mu,\uk}^{r}\neq0$. Moreover, this basis
becomes $\QQ$-rational in the associated graded $\oplus\text{Gr}_{2i}^{W}$.
\end{thm}
For the irrationality proofs, we need to apply this to certain 1-parameter
families.  We revise the notation for Laurent polynomials from $F_{\underline{\mathtt{a}}}$ to $\Phi$ to emphasize the specialization.
\begin{defn} 
\label{def2} A \emph{facile CY pencil} is a family of anticanonical hypersurfaces
$\mathrm{X}_{\xi}=\overline{\{\Phi(\xi)=0\}}\subset\PP_{\Diamond}$
parametrized by $\xi\in\PP^{1}$, where:\\
$\bullet$ $\Diamond\subset\RR^{n}$ is a facile polytope;\\
$\bullet$ $\Phi(\xi,\ux):=\sum_{\um\in\mathcal{A}_{\Diamond}}\xi^{o_{\um}}P_{\um}(\xi)\ux^{\um},$
with $o_{0}=0$ and $o_{\um}>0$ for $\um\neq\uo$, $P_{\um}(\xi)\in\ZZ[\xi]$
with $\gcd_{\um\in\mathcal{A}}\{P_{\um}(\xi)\}=1$, and (if $n>1$)
$\prod_{i=0}^{N}P_{\uv^{(i)}}(0)^{\ell_{i}^{(k)}}=1$ $(\forall k)$,
$P_{\uo}(0)=1$;\\
$\bullet$ the VHS on $H^{n-1}(\mathrm{X}_{\xi})$ is pure, with a
factor\footnote{``tr'' stands for ``transcendental'' (motivated by the $n=3$ case)} $H_{\text{tr}}^{n-1}(\mathrm{X}_{\xi})=:\mathcal{W}_{\Phi}$
with Hodge numbers $h^{n-1,0}=h^{n-2,1}=\cdots=h^{0,n-1}=1$; and\\
$\bullet$ $\sum_{i=1}^{N}o_{\uv^{(i)}}\ell_{i}(\un)\geq|\ul(\un)|^{+}$
for all $\un\in\ZZ_{\geq0}^{M}$ (automatic if $\ul(\un)$ effective).
\end{defn}
\noindent In particular, note that $\mathcal{W}_{\Phi}$ has maximal unipotent
monodromy at $\xi=0$. Replacing the $\{\mathtt{a}_i\}$ in $t_k=\underline{\mathtt{a}}^{\underline{\ell}^{(k)}}$ by the corresponding coefficients of $\Phi$, we get $t_{k}=\xi^{\sum o_{i}\ell_{i}^{(k)}}g_{k}(\xi)$
(with $g_{k}(\xi)\in1+\xi\ZZ[[\xi]]$) and $2\pi\mathbf{i}\tau_{k}=\log(t_k)=(\sum o_{i}\ell_{i}^{(k)})\log\xi+\sum_{m>0}\frac{h_{km}}{m}\xi^{m}$
(with $h_{km}\in\ZZ$). Substituting into $\eqref{5-1}$ and normalizing yields at once the
\begin{cor} \label{cor1}
Near $\xi=0$, any facile CY pencil admits a multivalued basis $\{\gamma_{j}\}_{j=0}^{n-1}$
of $\mathcal{W}_{\phi,\CC}^{\vee}\cong H_{n-1}^{\text{tr}}(\mathrm{X}_{\xi},\CC)$
$(\QQ$-rational in $\text{Gr}_{\bullet}^{W})$, holomorphic functions
$f^{(j)}(\xi)=\sum_{m\geq0}\mathfrak{f}_{m}^{(j)}\xi^{m}$, and an
integer $\varepsilon\in\mathbb{N}$, such that the periods
\[
\Pi_{\ell}(\xi):=\int_{\gamma_{\ell}}\varpi_{\xi}:=\int_{\gamma_{r}}\tfrac{1}{(2\pi\mathbf{i})^{n-1}}\mathrm{Res}_{\mathrm{X}_{\xi}}\left(\tfrac{dx_{1}/x_{1}\wedge\cdots\wedge dx_{n}/x_{n}}{\Phi(\xi,\ux)}\right)
\]
take the form 
\[
\Pi_{\ell}(\xi)=\frac{1}{(2\pi\mathbf{i})^{\ell}}\sum_{j=0}^{\ell}\frac{1}{j!}\log^{j}(\xi)f^{(\ell-j)}(\xi),
\]
with each $\varepsilon L_{m}^{j}\mathfrak{f}_{m}^{(j)}\in\ZZ$, and
$\mathfrak{f}_{0}^{(0)}=1$. 
\end{cor}
Note that in this scenario, $\gamma_{j}$ generates $\mathrm{Gr}_{2(j-n+1)}^{W}$,
the monodromy logarithm $N$ sends $\gamma_{n-1}\mapsto\gamma_{n-2}\mapsto\cdots\mapsto\gamma_{1}\mapsto\gamma_{0}$,
and $\gamma_{0}$ generates the local invariant cycles in $H_{n-1}^{\text{tr}}(\mathrm{X}_{\xi},\ZZ)$.

\section{Very special values}\label{S3}

Let $\phi\in\ZZ[x_{1}^{\pm1},\ldots,x_{n}^{\pm1}]$ ($n\geq1$) be
an integral Laurent polynomial, with reflexive Newton polytope $\Delta=\Delta_{\phi},$
and $\PP_{\Delta}$ be the (possibly singular) toric variety associated
to a maximal projective triangulation of $\Delta^{\circ}$ \cite{Ba}.
We begin by defining several notions we shall require for the general
irrationality statement.

Denote by $\X_{\phi}\overset{\rho}{\to}\PP^{1}$ the Zariski closure
in $\PP_{\Delta}\times\PP_{t}^{1}$ of $\mathcal{X}_{\phi}^{\times}:=\{1-t\phi(\ux)=0\}\subset\GG_{m}^{n}\times(\PP_{t}^{1}\backslash\{0\})$,
and write $t=\lambda^{-1}$, $X_{\phi,t}:=X_{\phi}^{\lambda}:=\rho^{-1}(t)$.
(Note that $\rho$ is given by $1/\phi$ resp. $\phi$ when working
in $t$ resp. $\lambda$.) We shall call $\phi$ \emph{involutive}
if there exists a birational map $\mathcal{I}:\X_{\phi}\dashrightarrow\X_{\phi}$
over $t\mapsto\pm t^{-1}$, defined over $\QQ$. Further, $\phi$
is said to admit a \emph{facile $r$-cover} ($r\in\NN$) if there
is a facile CY pencil $\mathrm{X}=\{\Phi(\xi,\uy)=0\}\subset\PP_{\Diamond}\times\PP_{\xi}^{1}$
and a dominant (generically $r:1$) rational map $\mathcal{J}:\mathrm{X}\dashrightarrow\X_{\phi}$
over $\xi\mapsto\xi^{r}=t$. We say that $\phi$ is of \emph{conifold
type} if%
\footnote{When $\phi$ is not $\Delta$-regular, we also assume that no \emph{non-generic}
singularities of $X_{\phi}^{\lambda}$ along the base locus $X_{\phi}^{\lambda}\cap\DD_{\Delta}$
occur for values $\lambda\in\D_{\phi}^{\times}$ (hence don't affect the
local system $\VV_{\phi}$), see below.%
} the nonzero critical values of $\phi:(\CC^{\times})^{n}\to\CC$ underlie
(isolated) ordinary double points in the fibers.

While we will not assume that $1-t\phi$ is $\Delta$-regular for
general $t$, or even that the generic fiber of $\rho$ is smooth,
we shall impose the conditions that the variation $\H_{\phi}$ underlain
by $R^{n-1}\rho_{*}\QQ$ is pure (of weight $(n-1)$),%
\footnote{For $n=1$, the generic $X_{\phi,t}$ is a pair of points, and $H^{0}$ {[}resp.
$H_{0}${]} means everywhere the augmentation cokernel {[}resp. kernel{]},
i.e. reduced (co)homology. Involutivity and reflexivity imply that
$x\phi(x)$ is a quadratic polynomial with two distinct roots.%
} and that the minimal sub-VHS $\V_{\phi}$ containing the class of
\[
\omega_{\phi,t}:=\omega_{\phi}^{\lambda}:=\tfrac{1}{(2\pi\mathbf{i})^{n-1}}\text{Res}_{X_{\phi}^{\lambda}}\left(\tfrac{dx_{1}/x_{1}\wedge\cdots\wedge dx_{n}/x_{n}}{1-t\phi(\ux)}\right)\in\Omega^{n-1}(X_{\phi}^{\lambda})
\]
is of rank $n$, with Hodge numbers $\{h^{i,n-i-1}\}_{0\leq i\leq n-1}$
all $1$. In this case we call $\phi$ \emph{principal} (since $\V_{\phi}$
is a ``principal VHS'' \cite{Ro}). Note that $\V_{\phi}$ is polarized by taking cup products on a (fiberwise) desingularization; write $Q$ for this polarization.

Our $\V_{\phi}$ is thus a polarized VHS on $\PP^1\setminus \D_{\phi}$, where
$\D_{\phi}$ denotes its singularity (discriminant) locus.  In particular, since $X_{\phi,0}=\DD_{\Delta}$,
$\V_{\phi}$ has maximal unipotent monodromy at $t=0$, and so $0\in\D_{\phi}$.
We write $\D_{\phi}^{\times}:=\D_{\phi}\cap\GG_{m}$, $P_{\phi}(t):=\prod_{t_0 \in\D_{\phi}^{\times}}\left(\tfrac{t}{t_0}-1\right)$,
$\delta_{\phi}:=|\D_{\phi}^{\times}|=\deg(P_{\phi})$, and $$\mathfrak{r}_{\phi}=\min\{|t_0|\mid t_0\in\D_{\phi}^{\times}\} .$$
For $\phi$ of conifold type, the monodromy of $\V_{\phi}$ about
each point of $\D_{\phi}^{\times}$ is given by a single Picard-Lefschetz
transformation. Since the local exponents are then given by $\{0,1,\ldots,n-2\}\cup\{\tfrac{n-2}{2}\}$,
an easy calculation%
\footnote{With no information on exponents, one would have $P_{\phi}(t)\delta_{t}^{n}+\sum_{k=0}^{n-1}\frac{F_{k}(t)}{P_{\phi}(t)^{n-k-1}}\delta_{t}^{k}$
($F_{k}\in\CC[t]$); the existence of holomorphic solutions with orders
$0$ thru $n-2$ about each root $t_0$ of $P_{\phi}$ forces
$(t-t_0)^{n-k-1}\mid F_{k}$. Note that if $n$ is even,
the exponent $\frac{n-2}{2}$ is repeated.%
} shows that the Picard-Fuchs operator for $\omega_{\phi}$ takes the
form
\[
D_{\phi}(t)=P_{\phi}(t)\delta_{t}^{n}+\sum_{k=0}^{n-1}f_{\phi,k}(t)\delta_{t}^{k}\in\CC[t,\delta_{t}].
\]
For $n>2$ we remark that $\rho$ is not semistable over $0$ (and
possibly at other points of $\D_{\phi}$) without blowing up $\PP_{\Delta}$
along the singularities of the base locus, but this won't be an issue
for us.

The natural isomorphism $\X_{\phi}^{\times}\to\GG_{m}^{n}$ provides $n$ coordinates
$x_{i}\in\mathcal{O}^{*}(\X_{\phi}^{\times})\cong H_{\mathcal{M}}^{1}(\X_{\phi}^{\times},\QQ(1))$
whose cup product produces the \emph{coordinate symbol} $$\{\ux\}=\{x_{1},\ldots,x_{n}\}\in H_{\M}^{n}(\X_{\phi}^{\times},\QQ(n)).$$
We shall call $\phi$ \emph{tempered} if this lifts to a motivic cohomology
class $$\Xi\in H_{\M}^{n}(\X_{\phi}^{-},\QQ(n)),$$ where $\X_{\phi}^{-}:=\X_{\phi}\setminus X_{\phi}^{\infty}$.
(See Theorem \ref{th3}(a) for a down-to-earth temperedness criterion when $n\leq 3$.) Writing $T_{x_{i}}$ {[}resp. $T_{\{\ux\}}${]} for the analytic chain
$x_{i}^{-1}(\RR_{-})$ {[}resp. $\cap_{i=1}^{n}x_{i}^{-1}(\RR_{-})${]}
and $U_{\phi}:=\PP^{1}\setminus\rho(T_{\{\ux\}})$, we term $\phi$
\emph{strongly tempered} if the \emph{higher normal function (HNF)} 
defined by 
\[
\nu_{\phi,t}:=\mathrm{AJ}_{X_{\phi,t}}^{n,n}(\Xi|_{X_{\phi,t}})\in H^{n-1}(X_{\phi,t},\CC/\QQ(n))
\]
has a single-valued holomorphic family of lifts
\[
\tilde{\nu}_{\phi,t}\in H^{n-1}(X_{\phi,t},\CC)
\]
over $U_{\phi}$.%
\footnote{Note that if we assume this only over $U_{\phi}\setminus U_{\phi}\cap\D_{\phi}$,
the lift extends to $U_{\phi}$ anyway: the single-valuedness of $\tilde{\nu}$
on a punctured disk about $t_{0}$ means $\nu$ has no singularity
at the center, so that $\tilde{\nu}$ uniquely extends to the whole
disk. The value $\tilde{\nu}(t_{0})$ lies in $\ker(T-I)$ in the
canonical extension $\H_{\phi}^{\text{lim}}$, which contains the
image of $H^{n-1}(X_{\phi,t_{0}})$; see \cite[$\S$5]{7K}.%
} Here by a ``lift'', we simply mean that the natural map $H^{n-1}(X_{\phi,t},\CC)\twoheadrightarrow H^{n-1}(X_{\phi,t},\CC/\QQ(n))$ maps $\tilde{\nu}_{\phi,t} \mapsto \nu_{\phi,t}$.

\begin{rem}
As defined, $\nu_{\phi}$ gives an element of the space $\mathrm{ANF}(\H_{\phi}(n))$
of admissible normal functions \cite[$\S$5.2]{7K}.
We shall only really care about the $\V_{\phi}$-component of $\nu_{\phi}$
below, so one might as well regard it (by projection) as an element of $\mathrm{ANF}(\V_{\phi}(n))$.
Note that any lift (such as $\tilde{\nu}_{\phi}$) of this component
must have nontrivial monodromy on $\PP^{1}\setminus\D_{\phi}$ (as
opposed to just $U_{\phi}$): otherwise its topological invariant
$[\nu_{\phi}]\in\mathrm{Hom}\left(\QQ(0),H^{n}\left(\rho^{-1}(\PP^{1}\setminus\D_{\phi}),\QQ(n)\right)\right)$
would vanish. Since the latter is computed by $\Omega_{\Xi}$ (see the proof of \cite[Cor. 4.6]{DK}), which
restricts to the Haar form on $\X_{\times}^{*}\cong\GG_{m}^{n}$, this
is absurd.
\end{rem}

Now either $\omega_{\phi,t}$ or $\tilde{\omega}_{\phi,t}:=t\omega_{\phi,t}$
is a holomorphic section of the canonically extended Hodge bundle
$\F_{e}:=\F^{n-1}\V_{\phi,e}$. As residue forms, they are really
most naturally regarded as elements of $H_{n-1}(X_{\phi,t},\CC)$
for any $t\in\PP^{1}$. If $t\notin\D_{\phi}$, then purity of $\H_{\phi}$
makes $H^{n-1}\to H_{n-1}$ an isomorphism,%
\footnote{Here and below, we use the polarization $Q$ to make this identification
(up to twist).%
} allowing us to treat them as cohomology classes; but over all of
$\PP^{1}$ or $U_{\phi}$, they only make sense as sections of $H_{n-1}$.
Conveniently enough, this pairs with $H^{n-1}(X_{\phi,t},\CC)$, allowing
us to define 
\[
\tilde{V}_{\phi}(\lambda):=\left\langle \tilde{\nu}_{\phi}^{\lambda},\tilde{\omega}_{\phi}^{\lambda}\right\rangle \in\mathcal{O}(U_{\phi}).
\]
This extends to a multivalued-holomorphic function
\[
V_{\phi}(\lambda):=\left\langle \nu_{\phi}^{\lambda},\tilde{\omega}_{\phi}^{\lambda}\right\rangle \;\;\text{mod }\tilde{\mathscr{P}}_{\phi}^{\QQ(n)}
\]
on $\PP^{1}\setminus\D_{\phi}$, defined modulo $\QQ(n)$-periods of
$\tilde{\omega}_{\phi}^{\lambda}$; we shall refer to both $V$ and
$\tilde{V}$ as the \emph{truncated higher normal function (THNF)}
associated to $\Xi$ and $\tilde{\omega}_{\phi}$. Note that $\tilde{V}_{\phi}$
\emph{cannot} extend to an entire function, since $\tilde{\nu}_{\phi}$
has nontrivial monodromy on $\PP^{1}\setminus\D_{\phi}$ and $\tilde{\omega}_{\phi}$
has $n$ independent periods (courtesy of the maximal unipotent monodromy).

We can now state the main result of this section:
\begin{thm}
\label{th2} Let $\phi(x_{1},\ldots,x_{n})$ be an integral Laurent
polynomial, such that $\phi(-\ux)$ has all positive coefficients,
which is reflexive, involutive, principal, strongly tempered, of conifold
type, and admits a facile $r$-cover. Assume that $\delta_{\phi}=2$,
$D_{\phi}\in\ZZ[t,\delta_{t}]$, and $\mathfrak{r}_{\phi}<e^{-n}$.
Then $\tilde{V}_{\phi}(0)\notin\QQ^{\times}$.
\end{thm}

\subsection{Proof of Theorem \ref{th2}}\label{S3.proof}

\subsubsection*{Step 1: The power series}

By involutivity, the four points of $\D_{\phi}$ have $\lambda$-values
$0,\lambda_{0},\pm\lambda_{0}^{-1},\infty$, with $|\lambda_{0}|=\mathfrak{r}_{\phi}<1<|\lambda_{0}|^{-1}$.
(The choice of $\pm$  here matches that in the involution $t\mapsto \pm t^{-1}$.) Moreover, the $\ZZ$-local system $\VV_{\phi}$ underlying $\V_{\phi}$
has maximal unipotent monodromy at $\lambda=0$, with (rank 1) invariant
subsystem on the disk $D_{\mathfrak{r}_{\phi}}$ generated by a family
of $(n-1)$-cycles $\vf_{0}^{\lambda}$. Indeed, we may assume that
$\vf_{0}^{\lambda}=\mathcal{I}_{*}{}'\vf_{0}^{\mathcal{I}(\lambda)}$
where $\text{Tube}({}'\vf_{0})=\mathbb{T}_{n}:=\left\{ |x_{1}|=\cdots=|x_{n}|=1\right\} \subset\PP_{\Delta}\setminus X_{\phi}^{\mathcal{I}(\lambda)}$.
As the Hodge bundle $\F_{e}$ has degree 1 \cite[p. 504]{GGK}, and $\omega_{\phi}\in\Gamma(\PP^{1},\F_{e})$
has a zero at $\infty$, we have $(\mathcal{I}^{*}\omega_{\phi})/\omega_{\phi}=M\lambda^{-1}$
for some $M\in\QQ^{\times}$. But then $(\mathcal{I}^{*})^{2}\omega_{\phi}=\mathcal{I}^{*}\frac{M}{\lambda}\omega_{\phi}=\frac{M^{2}}{\lambda\mathcal{I}(\lambda)}\omega_{\phi}=\pm M^{2}\omega_{\phi},$
and since $(\mathcal{I}^{*})^{2}$ acts on $\VV_{\phi}$ (and $\VV_{\phi}^{\CC}$
is irreducible), we must have $M^{2}\in\ZZ^{\times}=\{\pm1\}$. This
forces $M=\pm1$ hence $\mathcal{I}^{*}\omega_{\phi}=\pm\tilde{\omega}_{\phi}$.
The holomorphic period at $\lambda=0$ is therefore\begin{flalign*}
A_{\phi}(\lambda ) &:= \int_{\vf_0^{\lambda}} \tilde{\omega}^{\lambda}_{\phi} = \int_{'\vf_0^{\mathcal{I}(\lambda)}} \mathcal{I}^* \tilde{\omega}^{\lambda}_{\phi}\\
&= \pm \int_{'\vf_0^{\mathcal{I}(\lambda)}} \omega_{\phi}^{\mathcal{I}(\lambda)} = \pm \int_{'\vf_{0,\pm\lambda}} \omega_{\phi,\pm\lambda} \\
&= \pm\int_{\mathbb{T}_n}\tfrac{dx_1/x_1 \wedge \cdots \wedge dx_n/x_n}{1\mp \lambda \phi(\ux)} ,
\end{flalign*}and of course we may change the signs of $\vf_{0}$ and $\lambda$
if needed so that\begin{equation}\label{13-1}
A_{\phi}(\lambda)=\sum_{k\geq 0} [\phi^k]_{\uo} \lambda^k =: \sum_{k\geq 0}a_k \lambda^k
\end{equation}(where $[\cdot]_{\uo}$ denotes ``constant term''), with $a_{0}=1$.
By the fundamental result of \cite{DvdK}, we have $\mathfrak{r}_A := \left(\limsup_{k\to \infty} |a_k|^{1/k}\right)^{-1} \in \CC^{\times}$; and since $A_{\phi}$ is a period of $(\VV_{\phi},\tilde{\omega}_{\phi})$, $\mathfrak{r}_A$ must be the modulus of an element of $\D_{\phi}^{\times}$ (i.e., $\mathfrak{r}_{\phi}$ or $\mathfrak{r}_{\phi}^{-1}$).  Since the $[\phi^k]_{\uo}$ are all (positive) integers, $\mathfrak{r}_A = \mathfrak{r}_{\phi}$.

On the other hand, positivity of the coefficients of $\phi(-\ux)$ forces $\phi(T_{\{\ux\}})\subset [1,\infty]$. Now the global minimum of $\phi(T_{\{\ux\}})$ necessarily occurs at a critical point of $\phi$; more precisely, it may be regarded as the terminus of a Lefschetz thimble on the generator $\d T_{\Xi}$ of $\V_{\phi}^{0,\text{lim}}/N\V_{\phi}^{0,\text{lim}} \cong \text{Gr}^W_0 H_{n-1}(X_{\phi}^0)$, and thus is $\pm$ the $\lambda$-coordinate of a point of $\D_{\phi}^{\times}$. Clearly this must be the larger one, so that $U_{\phi}$ contains the disk $D_{\mathfrak{r}_{\phi}^{-1}}$ about $\lambda=0$. Writing \[ \tilde{V}_{\phi}(\lambda)=:\sum_{k\geq 0} v_k \lambda^k ,\]it follows that $\mathfrak{r}_V := \left(\limsup_{k\to\infty}|v_k|^{1/k}\right)^{-1}=\mathfrak{r}_{\phi}^{-1}$.\footnote{Recall that $\tilde{V}_{\phi}$ cannot be entire.} Moreover,  since $A(\lambda)$ is (up to scale) the \emph{only} period of $\tilde{\omega}_{\phi}$ invariant about $\lambda =0$, and it is \emph{not} invariant about $\lambda = \lambda_0$, we conclude that $\tilde{V}_{\phi}(\lambda)$ [resp. $\tilde{\nu}_{\phi}^{\lambda}$] is the unique analytic continuation of $V_{\phi}(\lambda)$ [resp. holomorphic lift of $\nu_{\phi}^{\lambda}$] which is well-defined on $U_{\phi}$.

In particular, $\tilde{V}_{\phi}(0)=v_0\in \CC$ is well-defined.  Assuming henceforth that $v_0 \neq 0$ (recall we are proving that $v_0\notin \QQ^{\times}$), we may consider \[B_{\phi}(\lambda):= -\tilde{V}_{\phi}(\lambda) + \tilde{V}_{\phi}(0)A_{\phi}(0) =: \sum_{k\geq 1} b_k \lambda^k.\]Obviously this has radius of convergence $\mathfrak{r}_B=\mathfrak{r}_A \,(<1)$, while $v_k = v_0 a_k - b_k \to 0$ as $k\to 0$. Restricting if necessary to a subsequence $a_{k_j} \to \infty$, we therefore have\[\lim_{j\to \infty}\frac{b_{k_j}}{a_{k_j}} = v_0.\]

\subsubsection*{Step 2:  The inhomogeneous equation}
Since the holomorphic period of $\tilde{\omega}$ about $\lambda = 0$ is obtained by substituting $\lambda$ for $t$ in the holomorphic period of $\omega$ about $t=0$, the same goes for the (inhomogeneous) Picard-Fuchs equations: that is,\[D_{\phi}(\lambda)\tilde{\mathscr{P}}_{\phi}^{\CC} = 0.\]Consequently, in $\V_{\phi}$ we have
\begin{equation}
P_{\phi}(\lambda) \nabla_{\delta_{\lambda}}^n [\tilde{\omega}_{\phi}] = -\sum_{k=0}^{n-1} f_{\phi,k}(\lambda) \nabla^k_{\delta_k} [\tilde{\omega}_{\phi}];
\end{equation}
and regarding $\tilde{\omega}_{\phi}$ as a section of $\V_{\phi,e}$, maximal unipotent monodromy at $\lambda = 0$ and $\text{Res}_0 \nabla = \tfrac{1}{2\pi \mathbf{i}} N$ imply the independence of the $\{ \nabla_{\delta_k}^k [\tilde{\omega}_{\phi}]\}_{k=0}^{n-1}$ in the fiber $\V_{\phi,e}^0$. As $N^n=0$, $\nabla_{\delta_{\lambda}}^n [\tilde{\omega}_{\phi}]$ vanishes in $\V_{\phi,e}^0$, and so all $f_{\phi,k}(0)=0$ $\implies$ $f_{\phi,k}(\lambda)=:\lambda g_{\phi,k}(\lambda)$.  Moreover, since $D_{\phi}(\lambda)\circ \lambda^{-1}$ (like $D_{\phi}(t)$) kills $[\omega_{\phi}]$, we have \[F(t)\circ D_{\phi}(t) = D_{\phi}(\lambda)\circ \lambda^{-1} = D_{\phi}(\pm t^{-1})\circ t\]for some function $F(t)$. Since by assumption $P_{\phi}(\lambda)$ is quadratic, a short computation shows that $F(t)=\pm\frac{1}{t}$ and that the $t g_{\phi,k}(t^{-1})$ must be polynomials, forcing the $g_{\phi,k}$ to be linear.

Next, define the \emph{Yukawa coupling} by
\begin{equation}\label{15-1}
Y(\lambda) :=Q\left( [\tilde{\omega}_{\phi}^{\lambda}],\nabla_{\delta_{\lambda}}^{n-1}[\tilde{\omega}^{\lambda}_{\phi}]\right) \in \CC(\lambda)^{\times}.
\end{equation}
Let $w$ be a local coordinate at a point $p\in\D_{\phi}^{\times}$, about which monodromy is described by $\vf\mapsto \vf-\langle \vf,\sigma\rangle\sigma$ for some vanishing cycle $\sigma$ (since $\phi$ is of conifold type).  Writing $[\tilde{\omega}_{\phi}^{\lambda}]$ with respect to a local basis $\{\sigma_0,\ldots,\sigma_{n-1}\}$ of $\VV_{\phi}$ with $\sigma_1,\ldots,\sigma_{n-1}$ invariant about $p$ and $\sigma = \sigma_0$ [resp. $\sigma_1$] for $n$ odd [resp. even] (with $\sigma_0\mapsto -\sigma_0$ resp. $\sigma_0 \mapsto \sigma_0 - \sigma_1$), the coefficients $u_i (w)$ of $\{\sigma_1,\ldots,\sigma_n\}$ are holomorphic and we have (near $p$) $u_0(w)\sim w^{\frac{n-2}{2}}$ [resp. $u_0(w)\sim u_1(w)\log(w)\sim w^{\frac{n-2}{2}}\log(w)$]. From this we compute $Y\sim w^{\frac{n-2}{2}-(n-1)}\times w^{\frac{n-2}{2}} = w^{-1}$, so that $Y$ has (at worst) simple poles at $\D_{\phi}^{\times}$.  Moreover, the pairing \eqref{15-1} makes sense at $\lambda=0$ in $\V_{\phi,e}^0$, so that $Y(0)\neq \infty$; while $\mathcal{I}^* \tilde{\omega}^{\lambda}_{\phi} = \pm \lambda \tilde{\omega}^{\lambda}_{\phi}$ $\implies$ $Y(\mathcal{I}(\lambda))=\pm \lambda^2 Y(\lambda)$ $ \implies$ $\text{ord}_{\infty}(Y)\geq 2$. Since $Y$ is rational with only 2 simple poles, it can have only the double zero at $\infty$, and thus\footnote{One may also show that $g_{\phi,n-1}(\lambda)=\frac{n}{2} P_{\phi} ' (\lambda)$, but we won't need this.}\[ Y(\lambda) = \frac{Y(0)}{P_{\phi}(\lambda)}.\]

To evaluate $Y(0)$, and also anticipating Step 3, we extend $\vf_0^{\lambda}$ to a basis $\{\vf_j^{\lambda}\}_{j=0}^{n-1}$ of $(\VV_{\phi}^{\lambda,\QQ})^{\vee} \cong H_{n-1}(X_\phi^{\lambda},\QQ)$ on an angular sector of the punctured disk $D_{\mathfrak{r}_{\phi}}^{\times}$ satisfying $\vf_{n-j-1} = N^j \vf_{n-1}$ hence $\vf_j \in W_{2(j-n+1)}$ ($N=\log(T)$ the monodromy logarithm at $0$, and $W_{\bullet} = W(N)[n-1]_{\bullet}$).  From $NQ=-QN$ and $N^n = 0$, the $Q(\vf_j,\vf_k)$ are zero for $j+k<n-1$, and the $(-1)^j Q(\vf_j,\vf_{n-j-1})$ equal a common constant $Q_0^{-1}\in \QQ^{\times}$.  We may then modify $$\vf_j\mapsto \hat{\vf}_j = \vf_j + \sum_{i<j} \alpha_{ij}\vf_i \in W_{2(j-n+1)} \;\;\;\;\;(\alpha_{ij}\in \QQ)$$ so that $Q(\hat{\vf}_i,\hat{\vf}_{n-j-1}) = (-1)^i Q_0^{-1} \delta_{ij}$, with $N\hat{\varphi}_j = \hat{\varphi}_{j-1} + \sum_{i<j-1} \eta_{ij} \hat{\vf}_i$ ($\eta_{ij}\in\QQ$) and dual basis $\{\check{\vf}_j\}$ of $\VV_{\phi}^{\QQ}$.  (Note that $\check{\vf}_j \in W_{2(n-j-1)}\VV_{\phi}$, $-N\check{\vf}_j = \check{\vf}_{j+1} + \sum_{i>j+1} \eta_{ji} \check{\vf}_i$, while $Q(\check{\vf}_{n-i-1},\check{\vf}_j) = (-1)^i Q_0 \delta_{ij}$.) Writing locally \[ [\tilde{\omega}_{\phi}^{\lambda}] = \sum_{j=0}^{n-1} \left( \int_{\hat{\vf}_j} \tilde{\omega}_{\phi}^{\lambda} \right) \check{\vf}_j =: \sum_{j=0}^{n-1} \hat{\pi}_j (\lambda) \check{\vf}_j\]in $\V_{\phi}$, and $\pi_j(\lambda):=\int_{\vf_j} \tilde{\omega}_{\phi}^{\lambda}$, we have \[\pi_0(\lambda) = \hat{\pi}_0(\lambda) = A_{\phi}(\lambda)=:A^{(0)}_{\phi}(\lambda)\]and \[\hat{\pi}_j = \pi_j(\lambda) + \sum_{i<j} \alpha_{ij} \pi_i (\lambda);\]while in accordance with the monodromy properties of $\vf_j$,
\begin{equation}\label{17-1}
\pi_j(\lambda) = \sum_{k=0}^j \frac{\log^k (\lambda)}{(2\pi \mathbf{i})^k k!} A_{\phi}^{(j-k)}(\lambda)
\end{equation}
for some functions $A_{\phi}^{(\ell)}(\lambda)$ holomorphic on $D_{\mathfrak{r}_{\phi}}$.  Clearly, the limits $\lim_{\lambda\to 0} \log^{\ell}(\lambda) \left( \delta_{\lambda}^{n-1} \hat{\pi}_j \right) (\lambda)$ are zero for $j<n-1$, while we have $\lim_{\lambda\to 0} (2\pi\mathbf{i})^{n-1}\left( \delta_{\lambda}^{n-1}\hat{\pi}_{n-1}\right) (\lambda)=1=A^{(0)}_{\phi}(0)$; and so $Y(0)=(2\pi \mathbf{i})^{1-n}\times Q(\check{\vf}_0,\check{\vf}_{n-1})=\frac{\pm Q_0}{(2\pi\mathbf{i})^{n-1}}$.

Now the fiberwise restrictions $\Xi^{\lambda}:=\Xi|_{X^{\lambda}_{\phi}}$ ($\lambda\neq \infty$) have trivial class in $\text{Hom}_{\text{MHS}}\left(\QQ(0),H^n(X^{\lambda}_{\phi},\QQ(n))\right)$, so that $T_{\Xi}$ is a coboundary over any sufficiently small $B\subset \mathbb{A}^1_{\lambda}$. Since the regulator current $R_{\Xi}$ has $dR_{\Xi} = \frac{dx_1}{x_1}\wedge \frac{dx_n}{x_n} - (2\pi\mathbf{i})^n \delta_{T_{\Xi}}$ on $\X_{\phi}^-$, writing $T_{\Xi}|_{\rho^{-1}(B)}\equiv \d \Gamma_B$ (mod $\d\rho^{-1}(B)$) yields a current $\tilde{R}^B_{\Xi}:=R_{\Xi}|_{\rho^{-1}(B)} + (2\pi\mathbf{i})^n \delta_{\Gamma_B}$ with closed fiberwise pullbacks. This yields
\begin{equation}\label{18-1}
\nabla_{\delta_{\lambda}}\nu_{\phi}^{\lambda} = \nabla_{\delta_{\lambda}}[\tilde{R}_{\Xi^{\lambda}}^B ] = \left[ \left. \left( \tfrac{dx_1}{x_1}\wedge\cdots \wedge \tfrac{dx_n}{x_n} \,\lrcorner \,\widetilde{\delta_{\lambda}} \right)\right|_{X^{\lambda}_{\phi}}\right] = -(2\pi\mathbf{i})^{n-1}[\omega_{\phi}^{\lambda}].
\end{equation}
Using this (and $Q(\omega,\nabla_{\delta_{\lambda}}^k \tilde{\omega})=0$ for $k<n-1$), one computes that $\delta_{\lambda}^k V(\lambda) = \delta_{\lambda}^k Q(\nu,\tilde{\omega})= Q(\nu,\nabla_{\delta_{\lambda}}^k \tilde{\omega})$ for $k<n$ and $-(2\pi\mathbf{i})^{n-1}Q(\omega,\nabla_{\delta_{\lambda}}^{n-1}\tilde{\omega})+Q(\nu,\nabla_{\delta_{\lambda}}^n \tilde{\omega})$ for $k=n$. Using \eqref{15-1}, we have at once \[D_{\phi}(\lambda) V_{\phi}(\lambda)= \pm (2\pi \mathbf{i})^{n-1}\lambda Y(\lambda) P_{\phi}(\lambda) = \pm Q_0 \lambda.\]Notice that while $V_{\phi}(\lambda)$ is multivalued, the ambiguities are killed by $D_{\phi}(\lambda)$. Moreover, $\tilde{V}_{\phi}(\lambda)$ satisfies the same equation, so that $D_{\phi}(\lambda)A_{\phi}(\lambda)=0$ $\implies$ 
\begin{equation}\label{18-2}
D_{\phi}(\lambda)B_{\phi}(\lambda)=\mp Q_0 \lambda.
\end{equation}

\subsubsection*{Step 3: Arithmetic of coefficients}
By \eqref{13-1} and the integrality of $\phi$, we have $a_m \in \ZZ$ ($\forall m$). Expressing $b_m$ as $\frac{p_m}{q_m}$ with $p_m \in \ZZ$ and $q_m \in \NN$ relatively prime, we claim that for some fixed $\varepsilon_B \in \NN$,
\begin{equation}\label{18-3}
q_m \mid \varepsilon_B (m!)^n\;\;(\forall m).
\end{equation}
Indeed, if we write $P_{\phi}(\lambda)=c_n ' \lambda^2 + c_n '' \lambda +1$ and $g_{\phi,j}(\lambda) = c_j '\lambda + c_j ''$ (all $c_j ', c_j '' \in \ZZ$; $c_n ' =\pm 1$), substituting $\sum_{m\geq 1}b_m \lambda^m = B_{\phi}(\lambda)$ in \eqref{18-2} yields $b_1 = \pm Q_0 \in \QQ$ and the recurrence \[-m^n b_m = \left( \Sigma_{j=0}^n c_j ' (m-2)^j \right) b_{m-2} + \left( \Sigma_{j=0}^n c_j '' (m-1)^j \right) b_{m-1}.\]Taking $\varepsilon_B = q_1$, this establishes \eqref{18-3}.

We claim that, in addition (modifying $\varepsilon_B$ if necessary),
\begin{equation} \label{19-1}
q_m \mid \varepsilon_B L_{rm}^n \;\; (\forall m).
\end{equation}
To show this, we make use of the facile $r$-cover $\mathcal{J}$, which induces an isomorphism of VHS $\mathcal{J}^*\V_{\phi}\cong \W_{\Phi}$ hence of their extended Hodge bundles $\mathcal{O}_{\PP^1_{\xi}}(r)\cong \mathcal{J}^* \F_{e,\phi}^{n-1} \cong \F_{e,\Phi}^{n-1}$ over $\PP^1_{\xi}$, of which $\mathcal{J}^*\omega_{\phi,\xi^r}$ and \[\varpi_{\xi} := (2\pi\mathbf{i})^{1-n} \text{Res}_{\mathrm{X}_{\xi}}\left( \tfrac{dy_1/y_1\wedge \cdots \wedge dy_n/y_n}{\Phi(\xi,\uy)}\right)\]are sections. Since $\gcd_{\um\in \mathcal{A}_{\Diamond}}\{P_{\um}(\xi)\}=1$ (see Defn. \ref{def2}), $\varpi_{\xi}$ has no zeroes over $\mathbb{A}^1_{\xi}$, so that both sections share the divisor $r[\infty]$. Since $(\mathcal{J}_{\xi}^{\pm 1})_*$ exchanges the generators $\mathcal{I}_*(\vf_0)$ and $\gamma_0$ of the integral invariant cycles about $\xi =0$ and $\t=0$, and $\lim_{t\to 0} \int_{\mathcal{I}_*(\vf_0)}\omega_{\phi,t} = 1 = \lim_{\xi \to 0}\int_{\gamma_0} \varpi_{\xi},$ we find that $\varpi_\xi = \mathcal{J}^* \omega_{\phi,\xi^r}$. Via $\mathcal{I}^*\tilde{\omega}_{\phi}^{\xi^r} = \omega_{\phi,\xi^r}$, Corollary \ref{cor1}, and \eqref{17-1}, we therefore have ($0\leq \ell\leq n-1$) \[\hat{\pi}_{\ell} (\xi^r )= \int_{\hat{\vf}_{\ell}} \tilde{\omega}_{\phi}^{\xi^r} = \int_{\mathcal{J}_* \hat{\gamma}_{\ell}} \omega_{\phi,\xi^r} = \int_{\hat{\gamma}_{\ell}} \varpi_{\xi}\]for some $\hat{\gamma}_{\ell} = \sum_{k\leq \ell} (2\pi\mathbf{i})^{k-\ell}\beta_{k\ell}\gamma_k \in W_{2(\ell-n+1)} \W_{\Phi,\QQ}^\vee$ ($\beta_{k\ell}\in \CC$; $\beta_{\ell\ell}=r^{\ell}$). That is,
\begin{equation} \label{20-1}
\hat{\pi}_{\ell}(\xi^r) \;=\; \frac{1}{(2\pi \mathbf{i})^{\ell}}\mspace{-20mu} \sum_{\tiny\begin{array}{c} m\geq 0, \\ 0\leq j\leq k\leq \ell \end{array}} \mspace{-20mu}\frac{1}{j!}\,\mathfrak{f}_m^{(k-j)}\, \beta_{k\ell}\,\, \xi^m\, \log^j (\xi).
\end{equation}

Returning to the $\lambda$-disk for a moment, we have $[\tilde{\omega}_{\phi}^0]\equiv Q_0 \hat{\vf}_{n-1}$ in $H_{n-1}(X_{\phi}^0)$ so that $\tilde{\nu}_{\phi}^0 = Q_0^{-1}v_0 \check{\vf}_{n-1}$ in $H^{n-1}(X_{\phi}^0 )$.  This fixes the constant of integration, so that \eqref{18-1} gives
\begin{equation} \label{20-2}
\begin{split}
\tilde{\nu}_{\phi}^{\lambda} - Q_0^{-1} v_0 \check{\vf}_{n-1} &= \int_0 \left( \nabla_{\delta_{\lambda}} \nu_{\phi}^{\lambda}\right) \frac{d\lambda}{\lambda} \\
&= -(2\pi\mathbf{i})^{n-1} \int_0 [\omega_{\phi}^{\lambda}]\frac{d\lambda}{\lambda}\\
&= (2\pi\mathbf{i})^{n-1}\sum_{\ell=0}^{n-1} \int_0 \hat{\pi}_{\ell}(\lambda)d\lambda \;\check{\vf}_{\ell}
\end{split}
\end{equation}
in $\V_{\phi}$. Since $Q(Q_0^{-1}v_0 \check{\vf}_{n-1},\tilde{\omega}_{\phi}^{\lambda})=v_0 A_{\phi}(\lambda)$ and $Q(\tilde{\nu}_{\phi}^{\lambda},\tilde{\omega}_{\phi}^{\lambda})=\tilde{V}_{\phi}(\lambda)$, pairing \eqref{20-2} with $[\tilde{\omega}_{\phi}^{\lambda}]=\sum_{\ell'=0}^{n-1} \hat{\pi}_{\ell'}(\lambda) \check{\vf}_{\ell'}$ yields the key formula \[-B_{\phi}(\lambda) = (-2\pi\mathbf{i})^{n-1} Q_0 \sum_{\ell=0}^{n-1} (-1)^{\ell} \hat{\pi}_{n-\ell-1}(\lambda) \int_0 \hat{\pi}_{\ell}(\lambda)d\lambda.\]Substituting $\xi^r =\lambda$, this becomes \[\pm\sum_{\mu\geq 1}b_{\mu}\xi^{r\mu}=(2\pi\mathbf{i})^{n-1}Q_0 \sum_{\ell=0}^{n-1} (-1)^{\ell} \hat{\pi}_{n-\ell-1}(\xi^r)\int_0 \hat{\pi}_{\ell}(\xi^r)r\xi^{r-1}d\xi,\]and using $\int_0 \log^a(x)x^{b-1}dx=x^b \sum_{c=0}^a \frac{(-1)^{a-c}a!}{(a-c)! b^{c+1}}\log^{a-c}(x)$ together with \eqref{20-1}, \[=Q_0 r\sum_{m\geq 0}\xi^m \sum_{(*)_m}\frac{(-1)^{\ell+j'-i'}}{(j'-i')!j''!}\frac{\mathfrak{f}_{m''}^{(k''-j'')}\mathfrak{f}_{m'}^{(k'-j')}\beta_{k'',n-\ell-1}\beta_{k',\ell}}{(m'+r)^{i'+1}}\log^{j''+j'-i'}(\xi),\]where the $\sum_{(*)_m}$ (finite for each $m$) is over $m'+m''=m-r$, $0\leq \ell\leq n-1$, $0\leq j''\leq k''\leq n-\ell-1$, and $0\leq i'\leq j'\leq k'\leq \ell$. Since this plainly has to be a power series in $\xi^r$, the $\log^* \xi$ terms must cancel out (forcing $j''=0=j'-i'$), leaving us with
\begin{equation} \label{21-1}
=Q_0 r \sum_{m\geq 0}\xi^m \sum_{(*)_m} \frac{(-1)^{\ell}}{(m'+r)^{j'+1}}\mathfrak{f}^{(k'')}_{m''}\mathfrak{f}^{(k'-j')}_{m'} \beta_{k'',n-\ell-1}\beta_{k',\ell} .
\end{equation}
Let $\{\sigma_1,\ldots,\sigma_d\}$ ($\sigma_1=1$) be a basis of the $\QQ$-vector space $\mathscr{B}$ generated by all the products $\{\beta_{ij}\beta_{k\ell}\}$, and write $\beta_{ij}\beta_{k\ell} = \sum_{s=1}^d \mathfrak{q}_s^{ijk\ell}\sigma_s$. Since $\eqref{21-1} = \pm \sum b_{\mu} \xi^{r\mu}$ with $b_{\mu}\in \QQ$ (and $\mathfrak{f}_b^{(a)}\in\QQ$), the resulting ``$\sigma_s$-components'' of \eqref{21-1} vanish for $s>1$, while
\begin{equation}\label{21-2}
\pm b_{\mu} = Q_0 r \sum_{(*)_{r\mu}} \frac{(-1)^{\ell}}{(m'+r)^{j'+1}} \mathfrak{f}^{(k'')}_{m''}\mathfrak{f}^{(k'-j')}_{m'} \mathfrak{q}_1^{k'',n-\ell-1,k',\ell}.
\end{equation}
Choose $\epsilon\in\NN$ sufficiently large that all the $\{\epsilon Q_0 \mathfrak{q}_1^{ijk\ell}\}$ (a finite set) are integers.  Now $m'+r\leq r\mu$ $\implies$ $\frac{L_{r\mu}^{j'+1}}{(m'+r)^{j'+1}}\in\ZZ$, while Corollary \ref{cor1} $\implies$ $\varepsilon \mathfrak{f}_{m''}^{(k'')}L_{r\mu}^{k''}, \varepsilon \mathfrak{f}^{(k'-j')}_{m'} L_{r\mu}^{k'-j'}\in \ZZ$. Since in each term of RHS\eqref{21-2} we have $k''+k'-j'+j'-1=k''+k'+1\leq (n-\ell-1)+\ell+1=n$, multiplying the original $\varepsilon_B$ by $\epsilon\varepsilon^2$ gives $\varepsilon_B b_{\mu} L_{r\mu}^n \in \ZZ$, as desired.

\subsubsection*{Step 4: Irrationality of $v_0$}
Let $\Lambda_m := \text{gcd}(m!,L_{rm})$, so that by Step 3 we have $\varepsilon_B \Lambda_m^n b_m =: B_m \in \ZZ$ ($\forall m$). Writing $\mathcal{P}_m :=\{ p \text{ prime}\mid p\leq m\}$, set \[\pi(m)=|\mathcal{P}_m | \;\; \text{and} \;\; \chi(m):=\sum_{p\in\mathcal{P}_m} \log(p).\]Evidently \[e^{\chi(m)}\leq \Lambda_m \leq \prod_{p\in \mathcal{P}_m} p^{\lfloor \log_p (rm)\rfloor } \leq (rm)^{\pi(m)},\]hence
\begin{equation} \label{22-1}
e^{\chi(m)/m} \leq \Lambda_m ^{1/m} \leq (rm)^{\pi(m)/m} .
\end{equation}
By the Prime Number Theorem and its proof, $\frac{\pi(m)}{m} \sim \frac{1}{\log(m)}$ and $\lim_{m\to \infty} \frac{\chi(m)}{m}=1$. So the outer terms of \eqref{22-1} limit to $e$, and so does $\Lambda_m^{1/m}$.

Finally, suppose $v_0 = \frac{P}{Q}$ ($P\in\ZZ$, $Q\in \ZZ\setminus \{0\}$). Then 
\begin{flalign*}
\limsup_{m\to \infty} &\left| \varepsilon_B \Lambda_m^n a_m P - B_m Q\right|^{\frac{1}{m}} \\
&= \left( \lim_{m\to \infty} \Lambda_m^{1/m}\right)^n \cdot \lim_{m\to\infty} |\varepsilon_B Q|^{\frac{1}{m}}\cdot \limsup_{m\to \infty} |a_m v_0 - b_m |^{\frac{1}{m}} \\
&= e^n \cdot 1 \cdot \mathfrak{r}_{\phi}\;\;\;\;\; \text{(by Step 1)}\\
&<1 \;\;\;\;\;\;\;\;\;\text{(by assumption)},
\end{flalign*}
and for some sufficiently large $m$ we therefore have \[0<\left| \varepsilon_B \Lambda_m^n a_m P - B_m Q\right| <1 ,\]a contradiction. Q.E.D.

\subsection{Casting a wider net} \label{S3.last}
By \cite[Thm. 5.2]{7K}, we can think of $\tilde{\nu}_{\phi}^0$ as computing the extension of $\QQ(0)$ by $W_{-2n}H^{n-1}(X_{\phi}^0,\QQ(n))\cong \QQ(n)$ associated to $\Xi|_{X_{\phi}^0} \in H^n_{\M}(X_{\phi}^0,\QQ(0))$. Since $(\mathcal{I}^*\omega)|_{X_{\phi}^0}$ generates the top graded piece $\text{Gr}^W_0 H_{n-1}(X_{\phi}^0,\QQ)$, pairing with it sends the above $\QQ(n)$ isomorphically to $(2\pi\mathbf{i})^n \QQ$. So $\tilde{V}_{\phi}(0)$ realizes $\mathrm{AJ}\left( \Xi|_{X_{\phi}^0}\right)\in \CC/(2\pi\mathbf{i})^n\QQ$, which one interprets (arguing as in [op. cit.], at least for $n\leq 3$) as a Borel regulator value $r_{\text{Bor}}$ for $K_{2n-1}$ of the number field $k$ required to resolve $X_{\phi}^0$. Since this is just the field of definition of $\mathcal{I}$, the numbers $\tilde{V}_{\phi}(0)$ potentially appearing in Theorem \ref{th2} are limited (for $n\geq 2$) to $\zeta(n)$.

The first step in a generalization of this result would be to expand the notion of involutivity:
\begin{itemize}[leftmargin={*}]
\item drop the requirement $k=\QQ$,
\item replace $t\mapsto \pm t^{-1}$ by $t\mapsto \frac{1}{Ct}$ ($C\in\ZZ$), and
\item allow $\mathcal{I}$ to be a correspondence inducing an isomorphism between the $\QQ$-VHS $\V_{\phi}$ and its pullback.
\end{itemize}
Unfortunately, we have to pay for this expansion with a stronger bound:
\begin{prop} \label{prop}
Let $\phi$ be as in Theorem \ref{th2}, but with the weaker involutivity just described. Assume in addition that $\mathfrak{r}_{\phi} < \frac{e^{-n}}{|C|}$. Then $\tilde{V}_{\phi}(0)\notin\QQ$.
\end{prop}
The proof is a straightforward, but tedious, generalization of the above. One nice formal consequence is that $\mathcal{I}^* \omega = \frac{c}{\lambda} \omega $ for $c\in\CC^{\times}$ with $\frac{c^2}{C}=\pm 1$, so that $\tilde{V}_{\phi}(0) = \langle \tilde{\nu}_{\phi}^0 , \tilde{\omega}_{\phi}^0 \rangle = c^{-1} \langle \tilde{\nu}_{\phi}^0 ,\mathcal{I}^* \omega_{\phi,0}\rangle = c^{-1} r_{\text{Bor}}.$ For example, if $n=2$ and $\mathcal{I}$ is defined over $\QQ(\mathbf{i})$, then one expects $r_{\text{Bor}}\underset{\QQ^{\times}}{\sim} \mathbf{i}{G}$, where $G = L(\chi_{-4},2)$ is Catalan's constant. But then, one also expects $c\underset{\QQ^{\times}}{\sim}\mathbf{i}$, so that $\tilde{V}_{\phi}(0) \underset{\QQ^{\times}}{\sim} G$ --- a not-insignificant ``calibration'', as irrationality of $\mathbf{i}G$ and of $G$ are rather different things. Naturally, we don't have a proof of the latter, but we will briefly discuss some higher normal functions with $G$ as a special value below.

\begin{rem}\label{IJrem}
The function of $\mathcal{I}$ and $\mathcal{J}$ in the above proof is to match periods of a pullback of $\V_{\phi}$ with those of $\V_{\phi}$ resp. the facile family.  If one has such a matching by other means, there is obviously no need for the maps of varieties.
\end{rem}

Another natural way to relax the hypotheses is to permit the Newton polytope of $\phi$ to be non-reflexive, as long as $\uo$ remains its unique interior integral point. This ensures that, while $X_{\phi}^{\lambda}$ may not be Calabi-Yau, $h^{n-1,0}(X_{\phi}^{\lambda})$ remains $1$.  More significantly, one could abandon the ``principality'' constraint that Hodge numbers of $\V_{\phi}$ \emph{all} equal $1$, in order to make contact with results (such as \cite{Z2}) involving linear forms in more than one zeta value; two likely sources of interesting examples will be discussed in $\S$\ref{S5.4}.

\section{Low dimension} \label{S4}
In order to implement Theorem \ref{th2} in any specific cases, we must be able to check strong temperedness and compute $\tilde{V}_{\phi}(0)$.  For the temperedness, we begin by examining the toric boundary structure more closely.  (This part draws heavily on \cite[$\S\S$2-3]{DK}.)  Let $\phi$ and $\PP_{\Delta}$ be as in the first paragraph of $\S$\ref{S3}, with associated family $\X_{\phi} \underset{\rho}{\to}\PP^1$ of anticanonical hypersurfaces.

Consider the toric variety $\hat{\PP}_{\Delta}$ associated to $\Delta^{\circ}$ without the triangulation, with canonical blow-down morphisms $\PP_{\Delta} \underset{b}{\twoheadrightarrow} \hat{\PP}_{\Delta}$, $\DD_{\Delta} \underset{b}{\twoheadrightarrow} \hat{\DD}_{\Delta} = \hat{\PP}_{\Delta} \setminus \mathbb{G}_m^n$. Let $\sigma \subset \Delta$ be a codimension-$j$ face with ($(j-1)$-dimensional) dual $\sigma^{\circ}\subset \Delta^{\circ}$, and $\mathbb{G}_m^{n-j} \cong \hat{\DD}_{\sigma}^{\times} \subset \hat{\DD}_{\sigma}\subset \hat{\DD}_{\Delta}$ the corresponding (open resp. closed) codimension-$j$ stratum. The ($(n-j+1)$-dimensional) strata $\DD_{\sigma}^{\alpha} \subseteq b^{-1}(\hat{\DD}_{\sigma})$ correspond to the ($(j-i-1)$-dimensional) faces $\sigma_{\alpha}^{\circ} \subseteq \sigma^{\circ}$ in the triangulation of $\Delta^{\circ}$. A basis of the flag $\text{ann}(\sigma^{\circ})\subseteq \text{ann}(\sigma_{\alpha}^{\circ})\subseteq \ZZ^n$ produces $n-j+1$ toric coordinates $\{x_1^{\sigma},\ldots,x_{n-j}^{\sigma};y_1,\ldots,y_i\}$ on $\DD_{\sigma}^{\alpha,\times}$, with $b$ induced by forgetting the $\{y_i\}$. Writing $Y_{\phi}:=X_{\phi}^{\lambda} \cap \DD_{\Delta} \underset{b}{\twoheadrightarrow} \hat{Y}_{\phi}$, each\footnote{The $\{ \hat{Y}_{\sigma},Y_{\sigma}^{\alpha}\}$ need not be irreducible, as $\lambda - \phi$ is not assumed $\Delta$-regular for generic $\lambda$.}\[Y_{\sigma}^{\alpha} := Y_{\phi} \cap \DD_{\sigma}^{\alpha} = \cup_{i=1}^{\mu_{\sigma}}Y_{\sigma}^{\alpha,i} \underset{b}{\twoheadrightarrow} \cup_{i=1}^{\mu_{\sigma}} \hat{Y}_{\sigma}^i = \hat{Y}_{\sigma} =:\hat{Y}_{\phi} \cap \DD_{\sigma}\]is cut out of $\DD_{\sigma}^{\alpha}$ resp. $\hat{\DD}_{\sigma}$ by a ``face polynomial'' $\phi_{\sigma}(x_1^{\sigma},\ldots ,x_{n-j}^{\sigma})$ given (up to a shift) by rewriting the terms $\{ c\ux^{\um} \mid \um\in\sigma \}$ of $\phi$ in terms of $\ux^{\sigma}$.

A precondition for temperedness of $\phi$ (as defined above) is that the iterated residues of the coordinate symbol $\{\ux\}|_{(X_{\phi}^{\lambda})^{\times}}$ along strata $Y_{\sigma}^{\alpha}$ of $X_{\phi}^{\lambda}\setminus (X_{\phi}^{\lambda})^{\times}$ all vanish, which is equivalent to the vanishing of the $\{ \ux^{\sigma}\}$ on $\hat{Y}_{\sigma}^{\times}$:

\begin{defn}
$\phi$ is \emph{weakly tempered}\footnote{Note that this is the definition of \emph{temperedness} in \cite[$\S$3]{DK}. Our use of ``tempered'' here correlates to the property of ``$\{\ux\}$ completes to a family of motivic cohomology classes'' in [op. cit.].} if the symbols 
\begin{equation}\label{26-1}\{ x_1^{\sigma},\ldots ,x_{n-j}^{\sigma}\}|_{\hat{Y}_{\sigma}^{\times}} = 0 \in K_{n-j}^M (\bar{\QQ}(\hat{Y}_{\sigma}^i))\otimes \QQ\end{equation}
for all $1\leq j\leq n-1$, $\sigma\in \Delta(j)$, and $i=1,\ldots,\mu_{\sigma}$.
\end{defn}

This is a condition on the face polynomials $\phi_{\sigma}=\prod\phi_{\sigma,i}$.  For example, \eqref{26-1} holds:
\begin{itemize}[leftmargin={*}]
\item for $j=n-1$ if all edge polynomials are cyclotomic; and
\item for $j=n-2$ if (for all $2$-faces) the $\phi_{\sigma,i}$ are \emph{Steinberg} (i.e. $\phi_{\sigma,i}(x,y)=0$ makes $\{ x,y\}=0$ in $K_2$).
\end{itemize}
We will say that $\phi_{\sigma}$ is $\QQ$\emph{-Steinberg} if its factors are Steinberg and defined over $\QQ$.

For $n\geq 4$, it may not even be the case that weak temperedness implies the existence of fiberwise lifts $\Xi^{\lambda} \in\mathrm{CH}^n(\widetilde{X_{\phi}^{\lambda}},n)\otimes \QQ$ of $\{\ux\}|_{X_{\phi}^{\lambda,\times}}$ to a desingularization. This is always true for $n\leq 3$; for $n=4,5$ it holds if (for instance) all boundary strata are rational and defined over a totally real number field. (See \cite[$\S$3]{DK} for other refinements.) For particular families one can certainly check temperedness in any dimension. However, in the remainder of this section, we prefer to restrict to the case $n\leq 3$.

Let $\mathcal{K}_{\phi} :=\rho(T_{\{\ux\} })\subset \PP^1$ and $U_{\phi} := \PP^1 \setminus \mathcal{K}_{\phi}$ (as in $\S$\ref{S3}). We are interested in conditions under which not only is $\phi$ strongly tempered, but where the (unique) single-valued lift $\tilde{\nu}_{\phi}^{\lambda}$ over $U_{\phi}$ is given by fiberwise restrictions of the regulator current \[R_{\{\ux\} } = \sum_{j=1}^n (-1)^{(n-1)(j-1)} (2\pi \mathbf{i})^{j-1}\log(x_j) \frac{dx_{j+1}}{x_{j+1}} \wedge \cdots \wedge \frac{dx_n}{x_n} \delta_{T_{x_1} \cap \cdots \cap T_{x_{j-1}}}\]up to push-fowards of currents from $Y_{\phi}$ and coboundary currents. The biggest nuisance turns out to be the correction term $(2\pi\mathbf{i})^n\delta_{\Gamma^{\lambda}}$ added to $R_{\Xi^{\lambda}}$ to produce a closed current; part of this correction is (the current of integration over) a chain on $X_{\phi}^{\lambda}$ bounding on an unknown $(n-2)$-cycle on $Y_{\phi}$. In order that this not contribute to $\tilde{V}_{\phi}$, we have to assume that $H_{n-2}$ of $Y_{\phi}$ (or part of it) vanishes.\footnote{To get a feel for this condition, and the proof that follows, one may consult \cite[$\S$4]{BKV} where a similar argument is run on a concrete example.}

To make these conditions relatively weak, we define some ``bad'' subsets of the boundary.  Let $\mathscr{I}_{\phi}\subset Y_{\phi}$ be the generic $\Delta$-irregularity locus of $\lambda -\phi$ --- that is, the closure of the union of all singularities and nonreduced components of all $Y_{\sigma}^{\times}$ (computable by taking partials of $\phi_{\sigma}$). Let $\mathscr{A}\subset \mathscr{I}$ be the locus of generic singularities of $X_{\phi}^{\lambda}$ ($\mathscr{A}$ is empty for $n=1$ or $2$). Denoting by $\mathbb{I}_{\Delta}$ the intersection with $\DD_{\Delta}$ of the closure of the locus $\cup_{i=1}^n \{ x_i = 1\}$ in $\PP_{\Delta}$, we write $\mathscr{J}_{\phi}$ for the union of all components of $Y_{\phi}$ not contained in $\mathbb{I}_{\Delta}$, and (for $n=3$) not of ``Steinberg type'' $\{ x_{i_1}+ x_{i_2} = 1,\, x_{i_3} = 0 \text{ or }\infty\}$.

\begin{thm}\label{th3}
Let $\phi \in \ZZ[x_1^{\pm1},\ldots ,x_n^{\pm1}]$ be a reflexive Laurent polynomial for $n=1$, $2$, or $3$.
\begin{enumerate}[label=(\alph*)]
\item Assume that edge polynomials \textup{[}resp. 2-face polynomials\textup{]} of $\phi$ are cyclotomic \textup{[}resp. $\QQ$-Steinberg\textup{]}, and that $\mathscr{I}\subset \mathbb{I}_{\Delta}$ and $\mathscr{I}\cap \mathscr{J}\subset \mathscr{A}$; if $n=3$, assume $\mathscr{A}$ consists only of $A_1$ singularities and lies in $\mathbb{I}_{\Delta}\cap \text{sing}(\DD_{\Delta})$.  Then $\phi$ is tempered.
\item If $n=2$, assume in addition that $\mathscr{J}$ is one point.  If $n=3$, assume in addition that $H_1(\mathscr{J}\setminus \mathscr{J}\cap \mathscr{A})=\{0\}$. Then $\phi$ is strongly tempered, with single-valued THNF
\begin{equation}\label{28-1} \tilde{V}_{\phi}(\lambda) = \int_{X_{\phi}^{\lambda}} R\{\ux\} \wedge \tilde{\omega}_{\phi}^{\lambda} \end{equation}
over $U_{\phi}$.
\end{enumerate}
\end{thm}

\begin{proof}[Sketch]
(a) is a restatement of part of \cite[Thm. 3.8]{DK}. The strong-temperedness part of (b) recapitulates \cite[Prop. 4.12]{DK}. We now show that \eqref{28-1} is consistent with the lift constructed in the proof of [loc. cit.]. Assuming first that $\mathscr{A}=\emptyset$, we can extend $\{\ux\}$ to a closed precycle on $Z^n(\X_{\phi}^-\setminus \mathscr{J}\times \mathbb{A}^1,n)$ by taking Zariski closure and (for $n=3$) adding precycles of the form $\left( t,1-t,1-\tfrac{x_{i_1}}{t}\right)_{t\in\PP^1}$ over the ``Steinberg type'' locus $\mathscr{S}$. As in [loc. cit.], one then has $\gamma\in Z^n(\X_{\phi}^- \setminus\mathscr{J}\times \mathbb{A}^1,n+1)$ and (closed) $\Xi\in Z^n(\X_{\phi}^-,n)$ with $\Xi|_{\X_{\phi}^-\setminus \mathscr{J}\times\mathbb{A}^1} = \{\ux\}+\imath^{\mathscr{S}}_*Z + \d\gamma$ $\implies$ $R_{\Xi}|_{\X_{\phi}^-\setminus \mathscr{J}\times\mathbb{A}^1} = R_{\{\ux\}} + \tfrac{1}{2\pi\mathbf{i}}dR_{\gamma} +\imath^{\mathscr{S}}_* R_Z - (2\pi \mathbf{i})^n\delta_{T_{\gamma}}$ and $T_{\Xi} \equiv \d \overline{T_{\gamma}} + \tau$ mod $\rho^{-1}(\mathcal{K})$, with $\tau$ supported on $\mathscr{J}\times (\PP^1,\mathcal{K})$. Arguing as in [loc. cit.], our hypotheses give that $\tau \underset{\text{hom}}{\equiv} 0$ on $\mathscr{J}\times(\PP^1,\mathcal{K})$. Writing this as $\tau\equiv\d C_{\mathscr{J}}$, the ``lift'' of $R_{\Xi}$ given there is $R_{\Xi}' = R_{\Xi} + (2\pi\mathbf{i})^n \delta_{\overline{T_{\gamma}}+C_{\mathscr{J}}}$, so that \[ \left. R_{\Xi} ' \right|_{\X_{\phi}^- \setminus \mathscr{J}\times\mathbb{A}^1} = R_{\{\ux\}} + \tfrac{1}{2\pi\mathbf{i}} dR_{\gamma}.\]Writing $X_{\phi,\epsilon}^{\lambda}$ for $X_{\phi}^{\lambda}$ minus a small tubular neighborhood of $\mathscr{J}$, this leads to \[\tilde{V}_{\phi}(\lambda) = \int_{X_{\phi}^{\lambda}} R_{\Xi}'\wedge \tilde{\omega}_{\phi}^{\lambda} = \lim_{\epsilon \to 0} \left\{ \int_{X_{\phi,\epsilon}^{\lambda}} R_{\{\ux\}}\wedge \tilde{\omega} + \tfrac{1}{2\pi\mathbf{i}} \int_{\d X_{\phi,\epsilon}^{\lambda}} R_{\gamma}\wedge\tilde{\omega} \right\} \]hence \eqref{28-1}, provided $\lim_{\epsilon\to 0}\int_{\d X_{\phi,\epsilon}^{\lambda}} R_{\gamma}\wedge \tilde{\omega}=0.$ Viewing $\gamma$ as a family of curves in $\square^{n+1}$ over $\X_{\phi}$, the $(n-2)$-current $R_{\gamma}$ is the push-forward of $R_{\{\uz\}}$ ($\uz = $ coordinates on $\square^{n+1}$). If $n=2$, and $n$ is a local holomorphic coordinate about $\mathscr{J}$, one checks that $R_{\gamma}$ is locally $\mathcal{O}(\log^2 u)$; clearly $\oint_{|u|=\epsilon}\log^2 (u)du\to 0$. For $n=3$ (with $|(u)_0|=\mathscr{J}$ locally), $R_{\gamma}$ is the current of integration over a 3-chain times a locally-$\mathcal{O}(\log^2(u))$ function, with the same result. Finally, in the event that $\mathscr{A}$ is nonempty, and $\widetilde{\X_{\phi}^-}$ the blowup along $\mathscr{A}\times\PP^1$ (with exceptional divisor $\mathcal{E}$, one replaces all complexes (of higher Chow cycles and currents) on $\X_{\phi}^-$ by cone (double-)complexes for the morphism $\mathcal{E}\longrightarrow \mathscr{A}\amalg \widetilde{\X_{\phi}^-}$. The assumption that $H_1(\mathscr{J}\setminus\mathscr{J}\cap\mathscr{A})=\{0\}$ allows $C_{\mathscr{J}}$ to be drawn so as to avoid $\mathcal{E}$.
\end{proof}

\noindent This theorem is likely far from optimal, but suffices for the applications in the next section.

Turning to the computation of $\tilde{V}_{\phi}(0)$, we have the following specialization result:\footnote{This may be viewed an analogue of \cite[Cor. 5.3]{7K}, but on the level of lifts (i.e. as an exact equality of complex numbers).}

\begin{cor}\label{cor3}
For $\phi$ as in Theorem \ref{th3}(b), if $\phi(-\ux)$ has all positive coefficients, then\[\tilde{V}_{\phi}(0)=\int_{\psi} R_{\{\ux\}}|_{X_{\phi}^0}\]for any $(n-1)$-cycle $\psi\subset X_{\phi}^0 \setminus \mathscr{J}$ representing $Q_0 \hat{\vf}_{n-1} \in H_{n-1}(X_{\phi}^0)$, the class \textup{(}from $\S$\ref{S3.proof}, Step 2\textup{)} with intersection number $\pm1$ against ${\vf}_0$.\footnote{One should think of $\psi$  as a membrane ``stretched once around'' $X_{\lambda}^0$.}
\end{cor}

\begin{proof}[Sketch]
The additional hypothesis on $\phi$ ensures that $\lambda=0$ belongs to $U_{\phi}$. Now $Z^0 = \text{Res}_{X^0_{\phi}}(\mathcal{I}^*\Xi)=\mathcal{I}^*\text{Res}_{X_{\phi,0}}(\Xi) = \mathcal{I}^* Z_0$ has $\Omega_{Z^0} = \mathcal{I}^* \Omega_{Z_0} = (2\pi\mathbf{i})^{n-1} \mathcal{I}^* \omega_0 = (2\pi\mathbf{i})^{n-1}\tilde{\omega}^0$, while $T_{Z^0}$ is an $(n-1)$-cycle with $dR_{Z^0} = \Omega_{Z^0} - (2\pi\mathbf{i})^{n-1}\delta_{T_{Z^0}}$ $\implies$ $[T_{Z^0}]=[\tilde{\omega}^0]=Q_0\hat{\vf}_{n-1}$ in homology. So for $\psi$ as above, there exists an $(n-2)$-current $\mathcal{R}$ on $X_{\phi}^0$ with $d\mathcal{R} = \omega^0 - \delta_{\psi}$, with closed hence (by hypothesis on $\mathscr{J}$) exact restriction to $\mathscr{J}$. We may therefore assume that $\mathcal{R}$ (is of intersection type with respect to $\mathscr{J}$ and) pulls back to $0$ on $\mathscr{J}$, so that $\lim_{\epsilon\to 0}\int_{\d X_{\phi,\epsilon}^0} R_{\{\ux\}}\wedge \mathcal{R} = 0$. It now follows that $\tilde{V}_{\phi}(0) = \lim_{\epsilon \to 0} \int_{X_{\phi,\epsilon}^0} R_{\{\ux\}}\wedge \tilde{\omega}_{\phi}^0 = \lim_{\epsilon\to 0} \int_{X_{\phi,\epsilon}^0} R_{\{\ux\}} \wedge \delta_{\psi}$ as claimed.
\end{proof}

\section{Examples and near-examples}\label{S5}
Here we record some Laurent polynomials that satisfy the conditions of Theorems \ref{th2} and \ref{th3}, as well as a few which stray close enough to warrant attention.

\subsection{$n=1$}
Let $b\in \ZZ_{>0}$, $a=2b+1$, and set \[\phi(x):=-x+a+\frac{1-a^2}{4x}.\]$X_{\phi}^{\lambda}$ is a pair of points $\{ p_+^{\lambda},p_-^{\lambda}\}$ which are distinct unless $\lambda$ is a root of $P_a(\lambda):=\lambda^2 -2a\lambda + 1$, and $\X_{\phi}$ has involution\footnote{of course, it is its own facile 1-cover!}\[(x,\lambda)\;\longmapsto \; \left( \lambda^{-1}x + \tfrac{1}{2}(1-\lambda^{-1})(a+1),\lambda^{-1}\right).\]We have $p_{\pm}^{\lambda}=\tfrac{1}{2}(a-\lambda \pm P_a(\lambda)^{1/2})$, and the 0-form\[\tilde{\omega}_{\phi}^{\lambda} = \lambda^{-1}\omega^{\lambda}_{\phi} = \frac{\pm 1}{2\sqrt{P_a(\lambda)}}\;\;\;\text{on}\;\;\;p_{\pm}^{\lambda}\]has period \[A(\lambda)=\int_{p_+^{\lambda}-p_-^{\lambda}}\tilde{\omega}_{\phi}^{\lambda} = \frac{1}{\sqrt{P_a(\lambda)}}.\]The regulator 0-current $R_{\lambda}=\log(p_{\pm}^{\lambda})$ on $p_{\pm}^{\l}$, and so the HNF is\[\tilde{V}_{\phi}(\l) = \frac{1}{\sqrt{P_a(\l)}}\log (p_+^{\l}/p_-^{\l}).\]Since $\mathfrak{r}_{\phi}=a-\sqrt{a^2 - 1}<e^{-1}$ for all $b\geq 1$, we conclude by Theorem \ref{th2} that \[\tilde{V}_{\phi}(0)=\log\left( \frac{b+1}{b}\right)\notin\QQ.\]

\subsection{$n=2$}\label{SS5.2}
Let \[\phi(x_1,x_2):= x_1^{-1}x_2^{-1} (1-x_1)(1-x_2)(1-x_1-x_2);\]the picture\[\includegraphics[scale=0.5]{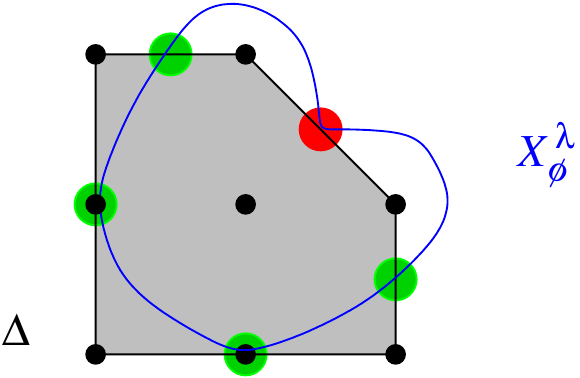}\]indicates the (reflexive) Newton polytope $\Delta$ and $X^{\lambda}_{\phi}\subset \PP_{\Delta}$. The green [resp. red] dots represent $Y_{\phi}\setminus \mathscr{J}$ [resp. $\mathscr{J}$]; edge polynomials are $x^{\sigma}-1$ or $(x^{\sigma}-1)^2$, so $\phi$ is strongly tempered by Theorem \ref{th3}(b).

The singular fibers of $\X_{\phi}$ are over $0$, $\infty$, and $t_{\pm} = \tfrac{-11\pm 5\sqrt{5}}{2}$, with $\mathfrak{r}_{\phi} = t_+ <e^{-2}$ (and $\X_{\phi,t_{\pm}}$ of Kodaira type $I_1$), while\[D_{\phi}(t) = (t^2 + 11t-1)\delta_t^2 + t(2t+11)\delta_t + t(t+3)\]is integral.\footnote{The coefficients of the holomorphic solution are the ``baby Ap\'ery'' sequence $a_m = \sum_{k=0}^m {\binom{m}{k}}^2 \binom{m+k}{k} = 1,\,3,\,19,\,147,\,\ldots .$} Involutivity is ensured by the (order 4) automorphism\[\mathcal{I}:\;(x_1,x_2,t)\longmapsto \left( \tfrac{x_1}{x_1 -1},\tfrac{1-x_2}{1-x_1-x_2},-\tfrac{1}{t}\right) ,\]while the facile 2-cover given by\[(1-\xi^2) + \xi y_1 + \xi^2 y_2 -\xi y_1 y_2 - \xi y_1^{-1} y_2^{-1} = 0\]maps down by\[\mathcal{J}:\; (y_1,y_2,\xi)\longmapsto \left( \tfrac{y_1}{\xi},\tfrac{y_1 - \xi}{y_1 y_2 -\xi},\xi^2 \right) .\]To compute the special value $\tilde{V}_{\phi}(0)$ using Corollary \ref{cor3}, note that $X^0 = \{ x_1 =1 \} \cup \{ x_2 =1 \} \cup \{ x_1 + x_2 = 1\}$, with $\psi$ going ``once around'' the figure. Furthermore, $R_{\{\ux\}} = \log(x_1 )\tfrac{dx_1}{x_1} - 2\pi\mathbf{i} \log (x_2)\delta_{T_{x_1}}$ vanishes on the first two components. Parametrizing the remaining component of $-\psi$ by $[0,1]\ni s \mapsto (1-s,s)$, we conclude (by Theorem \ref{th2}) that \[\tilde{V}_{\phi}(0) = -\int_0^1 \log(1-s)\tfrac{ds}{s} = \mathrm{Li}_2 (1) = \tfrac{\pi^2}{6} \notin \QQ.\]

\begin{rem}
In view of the results of Zagier's search for recurrencies of Ap\'ery type \cite{Za}, it seems likely that this $\phi$ is the unique example for $n=2$ that satisfies the conditions of Proposition \ref{prop}. One can match tempered Laurent polynomials (hence higher normal functions) to the sporadic examples of [op. cit.], but outside case ``D'' (just treated), both the bound on $\mathfrak{r}_{\phi}$ and involutivity fail. For instance, case ``E'' of [op. cit.] is $\phi(\ux) = (x_1 + x_1^{-1})(x_2 + x_2^{-1}) + 4$; this has $\tilde{V}_{\phi}(0) = 2G$ (Catalan), but $\mathfrak{r}_{\phi} C = \tfrac{1}{8}\cdot 32 = 4$ (too big), and the Kodaira fiber types at $\lambda_0$ and $C/\lambda_0$ (or $0$ and $\infty$) don't match, so that $\V_{\phi} \ncong \mathcal{I}^* \V_{\phi}$. This non-involutivity is not a problem for the approach via modular forms, which gives a different means for obtaining period expansions about any cusp; we are trading off this advantage for (at least in principle) the ability to treat non-modular families in higher dimension.
\end{rem}

\begin{rem}
One other ``near-example'' related to Catalan's constant arises from work of Zudilin \cite{Z3}, who found an Ap\'ery-like recurrence with rational solutions $a_m, b_m$ whose ratios $b_m/a_m$ converge rapidly to $G$. With some work, one can write the generating series $\sum_{m\geq 0} (a_m G - b_m) \lambda^m$ as a normal function associated to a higher cycle on a family of open\footnote{The corresponding Laurent polynomial is neither reflexive nor tempered; two points are removed from each fiber.} genus-9 curves, which are branched 4:1 covers of the ``baby Ap\'ery'' family of elliptic curves above! By construction, this has $V(0)=G$.
\end{rem}

\subsection{$n=3$}\label{SS5.3}
The Newton polytope $\Delta$ of
\begin{multline}\label{35-1}
\phi(x_1,x_2,x_3)=\\ x_1^{-1}x_2^{-1}x_3^{-1}(x_1 - 1)(x_2 - 1)(x_3 - 1)(1-x_1 -x_2 +x_1 x_2 -x_1 x_2 x_3)
\end{multline}
and its dual are \[\includegraphics[scale=0.5]{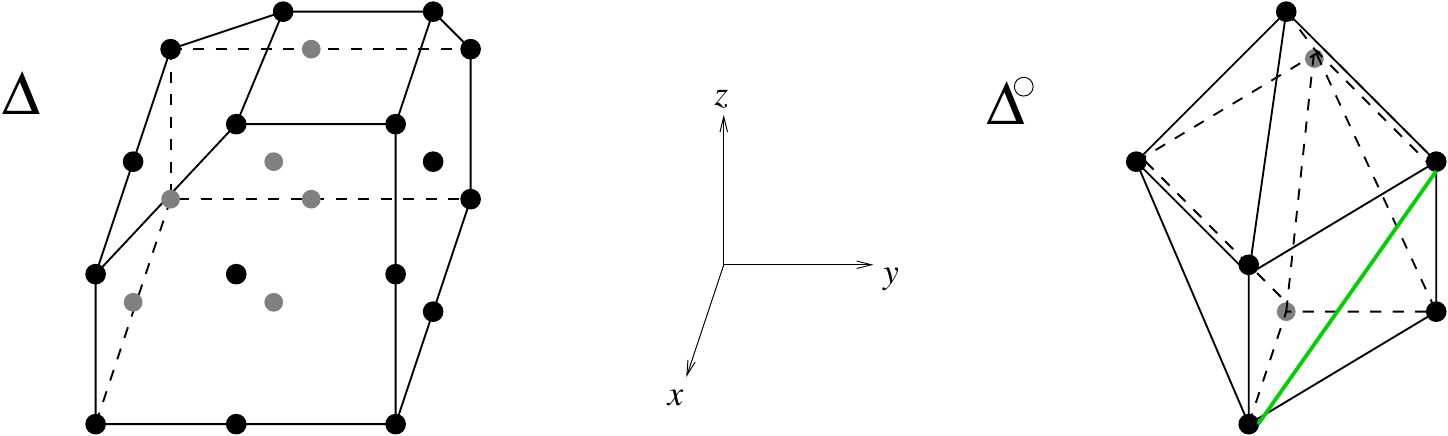}\]showing reflexivity; maximal triangulation adds the green edge. The edge and facet polynomials are products of $(x^{\sigma}_i - 1)$, $(x_i^{\sigma}-1)^2$, and $(1-x_1^{\sigma} \pm x_2^{\sigma})$, and $\mathscr{A}$ (red), $\mathscr{J}$ (blue), and $\mathscr{I}$ (green) are as depicted:\[\includegraphics[scale=0.5]{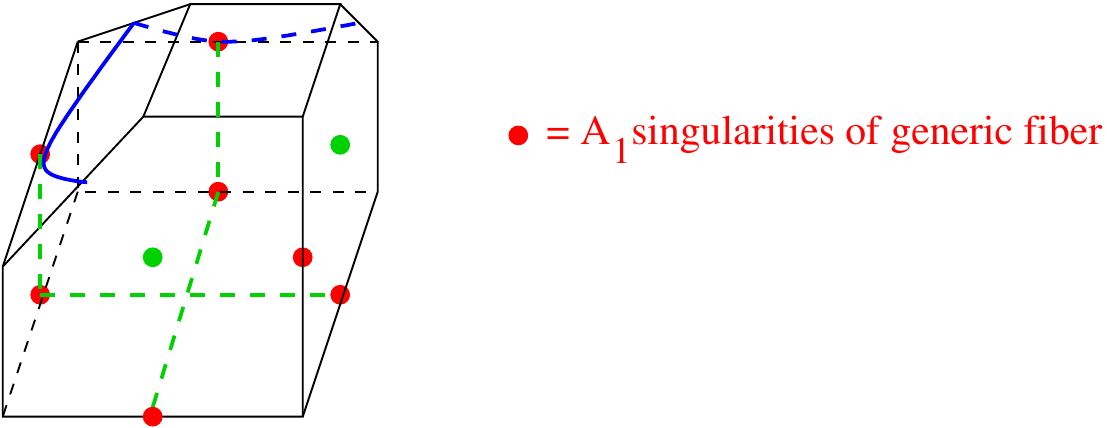}\]In particular, $\mathscr{J}\setminus\mathscr{J}\cap\mathscr{A}$ is two copies of $\mathbb{A}^1$ attached at a point, and we conclude by Theorem \ref{th3}(b) that $\phi$ is strongly tempered.

Singular fibers are at $0$, $\infty$, $t_{\pm}=(\sqrt{2}\pm 1)^4$, and $1$; the last of these does not contribute to monodromy of $\VV_{\phi}$, and so $\delta_{\phi} = 2$, while $\mathfrak{r}_{\phi}=t_- < e^{-3}$. To see that the generic Picard rank is $19$, one can use a torically-induced elliptic fibration (cf. \cite[$\S$2]{Ke}). The Picard-Fuchs operator is\[D_{\phi}(t) = (t^2 -34t+1)\delta_t^3 + 3t(t-17)\delta_t^2 + 3t(t-9)\delta_t + t(t-5),\]and the $\{a_m\}$ the famous Ap\'ery sequence $a_m = \sum_{k=0}^m {\binom{m}{k}}^2 {\binom{m+k}{k}}^2 = 1,5,73,1445,\ldots .$

Changing coordinates by $\sx_i = \tfrac{x_i}{x_i - 1}$ brings $1-t\phi(\ux)=0$ into the form studied by Beukers and Peters \cite{BP}. By the results of Peters and Stienstra \cite{PS}, $\X_{\phi,t}$ thus has a (facile) 2-cover by the \emph{Fermi family}\[\rx_{\xi} := \left\{ \xi \sum_{i=1}^3 (y_i + y_i^{-1}) + 1 + \xi^2 =0 \right\}.\]It also has an involution, by \[\mathcal{I}:\; (x_1,x_2,x_3,t)\longmapsto \left( \tfrac{x_3}{x_3 -1},\tfrac{(x_1 - 1)(x_2 -1 )}{1-x_1 -x_2 + x_1 x_2 -x_1 x_2 x_3 },\tfrac{x_1}{x_1 - 1},\tfrac{1}{t} \right) .\]

The 2-current \[R_{\{\ux\}} = \log(x_1)\tfrac{dx_2}{x_2} \wedge \tfrac{dx_3}{x_3} + (2\pi \mathbf{i})\log(x_2) \tfrac{dx_3}{x_3} \delta_{T_{x_1}} + (2\pi \mathbf{i})^2 \log(x_3) \delta_{T_{x_1}\cap T_{x_2}}\]vanishes on the components $\{ x_i = 1\}$ of $X_{\phi}^0$. The piece of $\psi$ on the remaining component $x_3 = \tfrac{(1-x_1)(1-x_2)}{x_1x_2}$ is parametrized by $0\leq r\leq s\leq 1\; \longmapsto \;(1-r,s,\tfrac{r(1-s)}{s(1-r)})$, and so (invoking Cor. \ref{cor3} and Thm. \ref{th2})
\begin{flalign*}
\tilde{V}_{\phi}(0) &= \int_{0\leq r\leq s\leq 1} \log(1-r)\text{dlog}(s)\wedge \text{dlog}(\tfrac{r}{1-r}) \\
&= -\int_0^1 \log(1-r)\log(r)\text{dlog}(\tfrac{r}{1-r}) \\
&= -2\int_0^1 \log(1-r)\log(r) \text{dlog}(r)\\
&= 2 \sum_{k\geq 1} \frac{1}{k} \int_0^1 r^{k-1} \log(r) dr \; \;= \;\;-2\zeta(3)\; \notin \;\QQ.
\end{flalign*}

\begin{rem}
In our normalization (see the definition of $b_m$ in $\S$\ref{S3.proof}, Step 1) one has $\{b_m\} = 0, -12, -\tfrac{351}{2}, -\tfrac{62531}{18},\,\ldots$, which is $-2$ times the usual second Ap\'ery sequence.  We note, however, that explicit knowledge of $\{a_m\}$ and $\{b_m\}$ is not needed for proving irrationality of $\zeta(3)$ via Theorem \ref{th2}.
\end{rem}

\begin{rem}
There are at least three ``near-examples'' for $n=3$, identified in \cite{dS} (and closely related to \cite{Go}), which satisfy all the criteria in Proposition \ref{prop} that we have checked, except for the bound: writing $\phi_{\rm{I}}$ for \eqref{35-1}, these are
\begin{flalign*}
\phi_{\rm{II}} &= (1-x_1 - x_2 -x_3)(1-x_1^{-1})(1-x_2^{-1})(1-x_3^{-1}) \\
\phi_{\rm{III}} &= (x_1 + x_2 +x_3)(-1+x_1^{^{-1}}\mspace{-5mu} + x_2^{^{-1}}\mspace{-5mu} + x_3^{^{-1}}\mspace{-5mu} - x_1^{^{-1}}x_2^{^{-1}} - x_1^{^{-1}}x_3^{^{-1}} - x_2^{^{-1}}x_3^{^{-1}}) \\
\phi_{\rm{IV}} &= (1-x_1 -x_2 -x_3)(1-x_1^{-1}-x_2^{-1}-x_3^{-1}) .
\end{flalign*}
They are reflexive, tempered, and involutive, with $C_{\rm{I}}=1$, $C_{\rm{II}}=16$, $C_{\rm{III}}=-27$, and $C_{\rm{IV}}=64$; while $\tilde{V}(0) \underset{\QQ^{\times}}{\sim} \zeta(3)$ except for $\phi_{\rm{III}}$, where $\tilde{V}(0) \underset{\QQ^{\times}}{\sim} L(\chi_{-3},3)\underset{\QQ^{\times}}{\sim} \pi^3 /\sqrt{3}$.  For $\mathfrak{r}_{\phi}|C|e^3$ we obtain $\approx 0.59$, $13.78$, $27.97$, resp. $80.34$, which satisfies the required bound ($<1$) only in the first case.
\end{rem}

\subsection{Higher dimension?}\label{S5.4}
Here we propose two sources for examples with $n\geq 4$, if one is prepared to weaken the hypotheses as in the last paragraph of $\S$\ref{S3.last}. In both cases, the Laurent polynomials considered, while not in general reflexive, all have Newton polytope $\Delta \subset [-1,1]^n$ having the origin as unique integer interior point. Details, proofs, and further developments will appear elsewhere.

Define the \emph{VZ polynomials} $\{\phi_n\}$ inductively by $\psi_1 = 1$,\[\psi_n (x_1,\ldots,x_n) :=x_1\cdots x_n + (1-x_n) \psi_{n-1}(x_1,\ldots,x_{n-1}),\]\[\phi_n(\ux) := (1-x_1^{-1})\cdots (1-x_n^{-1}) \psi_n (\ux) ;\]they are obtained (by substituting $\sx_i :=\tfrac{x_i}{x_i -1}$) from denominators of integrals first considered by Vasilyev \cite{Va} and Zudilin \cite{Z2} in their works on linear forms in zeta values. For $n=2$ and $3$, this recovers (up to inversion and permutation of coordinates) the Ap\'ery polynomials above. For $n=5$, we expect that $\phi_5$ is strongly tempered, and note that Hodge numbers of $\V_{\phi}$ are $(1,1,2,1,1)$. This means there are \emph{two} invariant periods $A(\lambda) = 1+\sum_{m\geq 1} a_m\lambda^m$, $B(\lambda) = \lambda +\sum_{m\geq 2} b_m \lambda^m$ about the maximal unipotent monodromy point, as $N$ has two primitive classes. Writing $V(\lambda)$ for the HNF, one then expects $\mathfrak{r}_{\phi}<e^{-5}$, and
\begin{flalign*}
C(\lambda) &:= -V(\l) + A(\l)V(0) + B(\l)\left( -V(0) A'(0) + V'(0)\right) \\
&= \sum_{m\geq 2} c_m \l^m
\end{flalign*}
to satisfy an inhomogeneous equation as above. Combining this with Vasilyev's results, one would conclude that\[V(\l)=\sum_{m\geq 0} \left( 2a_m \zeta(5) + b_m \zeta(3) - c_m \right) ,\]where $a_m,L_m^2 b_m,L_m^5 c_m \in \ZZ$, with the innocuous consequence that ``at least one of $\zeta(3)$ and $\zeta(5)$ is irrational.''

Recent work of F. Brown \cite{Br} provides another source of interesting Laurent polynomials. Given a permutation $\pi \in \mathfrak{S}_{n+3}$, write formally \[\theta_{\pi}(z_1,\ldots,z_{n+3}):= \prod_{i\in \ZZ/(n+3)\ZZ} (z_{\pi(i)}-z_{\pi(i+1)}),\]\[x_j^{\pi}:=-\text{CR}\left( z_{\pi(1)},z_{\pi(n+2-j)},z_{\pi(n+3-j)},z_{\pi(n+4-j)}\right),\]where $j=1,\ldots,n$, and \[\Omega_{\pi} :=\frac{dx_1^{\pi}}{x_1^{\pi}}\wedge \cdots \wedge \frac{dx_n^{\pi}}{x_n^{\pi}} ;\]if $\pi=\text{Id}$ then we drop the sub- and superscripts. Now let $\sigma \in \mathfrak{S}_{n+3}$ be a \emph{convergent} permutation in the sense of [op. cit.]; namely, we assume that for any $i\in \ZZ/(n+3)\ZZ$ and $2\leq k \leq n+1$, $\{ \sigma(i),\ldots,\sigma(i+k)\}$ is not a consecutive sequence of integers mod $(n+3)$. It turns out that $\theta_{\sigma}(\uz)/\theta(\uz)$ can be written as a Laurent polynomial $\phi_{\sigma}(x_1,\ldots,x_n)$ (with Newton polytope $\Delta_{\sigma}$), and the \emph{basic cellular integrals} on $\mathcal{M}_{0,n}$ of [op. cit.] become the integrals\[ I_{\sigma}^{(k)} := \int_{T_{\{\ux\}}}\phi_{\sigma}(\ux)^{-k} \Omega_{\sigma}\]on $\PP_{\Delta_{\sigma}}$. Defining $X_{\sigma}^{\l}\subset \PP_{\Delta_{\sigma}}$ by $\l=\phi_{\sigma}(\ux)$, the generating series 
\begin{flalign*}
V_{\sigma}(\l)&:= (2\pi\ay)^{1-n}\sum_{k\geq 0} I_{\sigma}^{(k)} \\
&=(2\pi\ay)^{1-n}\int_{T_{\{\ux\}}} \frac{dx_1/x_1\wedge \cdots \wedge dx_n/x_n}{\l-\phi_{\sigma}(\ux)}
\end{flalign*}
may be rewritten (using integration by parts) in the form \[=\int_{X_{\sigma}^{\l}} \left. R_{\{\ux\}} \right|_{X_{\sigma}^{\l}} \wedge \tilde{\omega}^{\l} ,\]which is a truncated HNF under a strong temperedness hypothesis. Finally, involutivity may be arranged via the additional hypothesis that $\sigma^{-1} = \pi_1 \circ \sigma \circ \pi_2$, with $\pi_1,\pi_2$ belonging to the dihedral group $D_{n+3}$.

\curraddr{${}$\\
\noun{Department of Mathematics, Campus Box 1146}\\
\noun{Washington University in St. Louis}\\
\noun{St. Louis, MO} \noun{63130, USA}}

\email{${}$\\
\emph{e-mail}: matkerr@math.wustl.edu}

\begin{thebibliography}{?????}
\bibitem[AESZ]{CYclass} G. Almkvist, C. van Enckevort, D. van Straten, and W. Zudilin, \emph{Tables of Calabi-Yau equations}, arXiv:math/0507430; database at http://www2.mathematik.uni-mainz.de/CYequations/db.

\bibitem[ASZ]{ASZ} G. Almkvist, D. van Straten and W. Zudilin, \emph{Ap\'ery limits of differential equations of order 4 and 5}, in ``Modular Forms and String Duality (Yui, Verrill, Doran Eds.)'', 105-123, Fields Inst. Commun. Ser. 54, Amer. Math. Soc. \& Fields Inst., 2008.

\bibitem[AGM]{AGM} P. Asipnwall, B. Greene and D. Morrison, \emph{The monomial-divisor mirror map}, Int. Math. Res. Notices (1993), 319-337.

\bibitem[BR]{BR} K. Ball and T. Rivoal, \emph{Irrationalit\'e d'une infinit\'e de valeurs de la fonction z\^eta aux entiers impairs}, Invent. Math. 146 (2001), 193-207.

\bibitem[Ba]{Ba} V. Batyrev, \emph{Dual polyhedra and mirror symmetry for Calabi-Yau hypersurfaces in toric varieties}, J. Algebraic Geom. 3 (1994), no. 3, 493-535.

\bibitem[Be]{Be} F. Beukers, \emph{Irrationality of $\pi^2$, periods of an elliptic curve and $\Gamma_1(5)$}, in ``Diophantine approximations and transcendental numbers (Luminy, 1982)'', 47-66, Prog. Math. 31, Birkh\"auser, Boston, 1983.

\bibitem[BP]{BP} F. Beukers and C. Peters, \emph{A family of $K3$ surfaces and $\zeta(3)$}, J. Reine Angew. Math. 351 (1984), 42-54.

\bibitem[BKV]{BKV} S. Bloch, M. Kerr and P. Vanhove, \emph{A Feynman integral via higher normal functions}, Compos. Math. 151 (2015), no. 12, 2329-2375.

\bibitem[BV]{BV} S. Bloch and M. Vlasenko, \emph{Gamma functions, monodromy, and Ap\'ery constants}, preprint, arXiv:1908.07501.

\bibitem[Br]{Br} F. Brown, \emph{Irrationality proofs for zeta values, moduli spaces and dinner parties}, Mosc. J. Comb. Number Theory 6 (2016), no. 2-3, 102-165.

\bibitem[Fano1]{Fano1} G. Brown, A Kasprzyk et al, ``Graded Ring Database'', at http://www.grdb.co.uk.

\bibitem[Fano2]{Fano} T. Coates, A. Corti, S. Galkin, V. Golyshev and A. Kasprzyk, ``Fano varieties and extremal Laurent polynomials'', research blog and database, http://www.fanosearch.net.

\bibitem[dS]{dS} G. da Silva Jr., \emph{On the arithmetic of Landau-Ginzburg model of a certain class of threefolds}, Commun. Number Theory Phys. 13 (2019), no. 1, 149-163.

\bibitem[7K]{7K} P. del Angel, C. Doran, J. Iyer, M. Kerr, J. Lewis, S. M\"uller-Stach, and D. Patel, \emph{Specialization of cycles and the $K$-theory elevator}, Commun. Number Theory Phys. 13 (2019), no. 2, 299-349.

\bibitem[DK]{DK} C. Doran and M. Kerr, \emph{Algebraic $K$-theory of toric hypersurfaces}, Commun. Number Theory Phys. 5 (2011), no. 2, 397-600.

\bibitem[DM]{DM} C. Doran and A. Malmendier, \emph{Calabi-Yau manifolds realizing symplectically rigid monodromy tuples}, Adv. Theor. Math. Phys. 23 (2019), no. 5, 1271-1359.

\bibitem[DvdK]{DvdK} J. Duistermaat and W. van der Kallen, \emph{Constant terms in powers of a Laurent polynomial}, Indag. Math. N.S. 9 (1998), 221-231.

\bibitem[Ga]{Ga} S. Galkin, \emph{Ap\'ery constants of homogeneous varieties}, arXiv:1604.04652.

\bibitem[GGI]{GGI} S. Galkin, V. Golyshev and H. Iritani, \emph{Gamma classes and quantum cohomology of Fano manifolds: gamma conjectures}, Duke Math. J. 165 (2016), no. 11, 2005-2077.

\bibitem[GI]{GI} S. Galkin and H. Iritani, \emph{Gamma conjecture via mirror symmetry}, in ``Primitive Forms and Related Subjects (Hori, Li, Li, Saito Eds.)'', 55-115, Mathematical Society of Japan, Tokyo, 2019.

\bibitem[GGK]{GGK} M. Green, P. Griffiths, and M. Kerr, \emph{Some enumerative global properties of varietions of Hodge structures}, Mosc. Math. J. 9 (2009), no. 3, 469-530.

\bibitem[Go]{Go} V. Golyshev, \emph{Deresonating a Tate period}, arXiv:0908.1458.

\bibitem[GZ]{GZ} V. Golyshev and D. Zagier, \emph{Proof of the gamma conjecture for Fano 3-folds of Picard rank 1}, Izv. Math. 80 (2016), no. 1, 24-49.

\bibitem[HLY]{HLY} S. Hosono, B. Lian and S.-T. Yau, \emph{GKZ-generalized hypergeometric systems in mirror symmetry of Calabi-Yau hypersurfaces}, Comm. Math. Phys. 182 (1996), no. 3, 535-577.

\bibitem[HLYZ]{HLYZ} A. Huang, B. Lian, S.-T. Yau and X. Zhu, \emph{Chain integral solutions to tautological systems}, Math. Res. Lett. 23 (2016), no. 6, 1721-1736.

\bibitem[Ir]{Ir} H. Iritani, \emph{Quantum cohomology and periods}, Ann. Inst. Fourier 61 (2011), no.7, 2909-2958.

\bibitem[IKSY]{Yo} K. Iwasaki, H. Kimura, S. Shinemura, and M. Yoshida, ``From Gauss to Painlev\'e:  A modern theory of special functions'', Aspects of Mathematics, E16, Friedr. Vieweg \& Sohn, Braunschweig, 1991.

\bibitem[Ke]{Ke} M. Kerr, \emph{Indecomposable $K_1$ of elliptically fibered $K3$ surfaces: a tale of two cycles}, in ``Arithmetic and geometry of $K3$ surfaces and C-Y threefolds (Laza, Sch\"utt, Yui Eds.)'', 387-409, Fields Inst. Comm. 67, Springer, New York, 2013.

\bibitem[KL]{KL} M. Kerr and J. Lewis, \emph{The Abel-Jacobi map for higher Chow groups, II}, Invent. Math. 170 (2007), 355-420.

\bibitem[KLM]{KLM} M. Kerr, J. Lewis, and S. M\"uller-Stach, \emph{The Abel-Jacobi map for higher Chow groups}, Compos. Math. 142 (2006), 374-396.

\bibitem[LZ]{LZ} B. Lian and M. Zhu, \emph{On the hyperplane conjecture for periods of Calabi-Yau hypersurfaces in $\PP^n$}, preprint, arXiv:1610.07125.

\bibitem[Na]{Na} M. Nair, \emph{On Chebyshev-type inequalities for primes}, Amer. Math. Monthly 89 (1982), no. 2, 126-129.

\bibitem[Pe]{Pe} C. Peters, \emph{Monodromy and Picard-Fuchs equations for families of $K3$-surfaces and elliptic curves}, Ann. Sci. \'Ecole Norm. Sup. (4) 19 (1986), no. 4, 583-607.

\bibitem[PS]{PS} C. Peters and J. Stienstra, \emph{A pencil of $K3$ surfaces related to Ap\'ery's recurrence for $\zeta(3)$ and Fermi surfaces for potential zero}, in ``Arithmetic of complex manifolds (Erlangen, 1988)'', 110-127, Lecture Notes in Math. 1399, Springer, Berlin, 1989.

\bibitem[Ro]{Ro} C. Robles, \emph{Principal Hodge representations}, in ``Hodge theory, complex geometry, and representation theory'', 259-283, Contemp. Math. 608, AMS, Providence, 2014.

\bibitem[vdP]{vdP} A. van der Poorten, \emph{A proof that Euler missed . . . Ap\'ery's proof of the irrationality of $\zeta(3)$. An informal report}, Math. Intellegencer 1 (1978/79), no. 4, 195-203.

\bibitem[Va]{Va} D. Vasilyev, \emph{On small linear forms for the values of the Riemann zeta-function at odd points}, Preprint no. 1 (558), Nat. Acad. Sci. Belarus, Institute Math., Minsk, 2001.

\bibitem[Za]{Za} D. Zagier, \emph{Integral solutions of Ap\'ery-like recurrence equations}, Groups and symmetries, 349-366, CRM Proc. Lecture Notes, 47, Amer. Math. Soc., Providence, RI, 2009.

\bibitem[Z1]{Z1} W. Zudilin, \emph{One of the numbers $\zeta(5),\zeta(7),\zeta(9),\zeta(11)$ is irrational} (Russian), Uspekhi Mat. Nauk 56 (2001), no. 4 (340), 149-150; translation in Russian Math. Surveys 56 (2001), no. 4, 774-776.

\bibitem[Z2]{Z2} W. Zudilin, \emph{Well-poised hypergeometric service for Diophantine problems of zeta values}, J. Th\'eor. Nombres Bordeaux 15 (2003), no. 2, 593-626.

\bibitem[Z3]{Z3} W. Zudilin, \emph{An Ap\'ery-like difference equation for Catalan's constant}, Electron. J. Combin. 10 (2003), Research Paper 14, 10 pp.

\end{thebibliography}
\end{document}